\documentclass[11pt]{article}
\usepackage[final]{showkeys}
\usepackage[obeyspaces,hyphens,spaces]{url}
\renewcommand*\showkeyslabelformat[1]{%
\fbox{\parbox[t]{1.4 cm}{\raggedright\normalfont\small\url{#1}}}}
\usepackage{amsthm,amsmath,amsfonts,mathtools,wasysym,empheq} 
\usepackage{amssymb} 
\usepackage{mathpazo}
\usepackage[mathscr]{eucal} 
\usepackage[top=2cm,bottom=2cm,left=2cm,right=2cm]{geometry}
\usepackage[title]{appendix}
\usepackage[table, dvipsnames]{xcolor} %
\usepackage[labelsep=space]{caption}  
\usepackage{graphicx, soul}
\usepackage{float,booktabs,multirow}
\restylefloat{table*}
\definecolor{labelkey}{rgb}{.1,.1,.8}
\definecolor{refkey}{rgb}{0,0.6,0.0}

\usepackage[shortlabels]{enumitem}
\setlist[enumerate]{itemsep=0.5pt,topsep=1pt}
\setlist[itemize]{itemsep=0.5pt,topsep=1pt}
\definecolor{dgreen}{rgb}{0.00,0.49,0.00}
\definecolor{dblue}{rgb}{0,0.08,0.75}
\colorlet{myblue}{dblue}
\colorlet{mygreen}{dgreen}
\definecolor{myfirstblue}{rgb}{.8, .8, 1}

\newcommand*\mybluebox[1]{%
\colorbox{RoyalBlue!20}{\hspace{1em}#1\hspace{1em}}}

\usepackage{hyperref}
\hypersetup{ 
colorlinks= true, 
urlcolor = magenta, 
citecolor = mygreen,
linkcolor = myblue,
}

\allowdisplaybreaks
\usepackage[capitalize]{cleveref} 
\crefdefaultlabelformat{#2\textup{#1}#3}
\crefrangeformat{enumi}{#3\textup{#1}#4\textup{\textendash}#5\textup{#2}#6}
\crefname{equation}{}{}

\crefname{chapter}{Appendix}{chapters}
\crefname{item}{}{items}
\crefname{figure}{Figure}{Figures}
\crefname{theorem}{\protect\theoremname}{Theorems}
\crefname{lemma}{\protect\lemmaname}{Lemmas}
\crefname{proposition}{\protect\propositionname}{Propositions}
\crefname{corollary}{\protect\corollaryname}{\protect\corollaryname}
\crefname{definition}{\protect\definitionname}{\protect\definitionname}
\crefname{fact}{\protect\factname}{\protect\factname}
\crefname{example}{\protect\examplename}{Examples}
\crefname{algorithm}{Algorithm}{Algorithms}
\crefname{remark}{\protect\remarkname}{\protect\remarkname}
\crefname{case}{\protect\casename}{\protect\casename}
\crefname{question}{\protect\questionname}{\protect\questionname}
\crefname{claim}{\protect\claimname}{\protect\claimname}
\crefname{enumi}{}{}
\crefname{appsec}{Appendix}{Appendices}

\makeatletter
\g@addto@macro\normalsize{%
  \setlength\abovedisplayskip{6pt}
  \setlength\belowdisplayskip{6pt}
  \setlength\abovedisplayshortskip{6pt}
  \setlength\belowdisplayshortskip{6pt}
}
\let\orgdescriptionlabel\descriptionlabel
\renewcommand*{\descriptionlabel}[1]{%
	\let\orglabel\label
	\let\label\@gobble
	\phantomsection
	\edef\@currentlabel{#1}%
	\let\label\orglabel
	\orgdescriptionlabel{#1}%
}

\let\leq\leqslant
\let\geq\geqslant
\renewcommand{\implies}{\Rightarrow}
\renewcommand{\iff}{\Leftrightarrow}

\def\th@plain{%
	\thm@notefont{} 
	\itshape 
}
\def\th@definition{%
	\thm@notefont{}
	\normalfont 
}
\g@addto@macro\th@remark{\thm@headpunct{}}
\g@addto@macro\th@definition{\thm@headpunct{}}
\g@addto@macro\th@plain{\thm@headpunct{}}
\makeatother

\theoremstyle{plain}
\newtheorem{theorem}{\protect\theoremname}[section]
\newtheorem{corollary}[theorem]{\protect\corollaryname}
\newtheorem{lemma}[theorem]{\protect\lemmaname}
\newtheorem{proposition}[theorem]{\protect\propositionname}

\theoremstyle{definition}
\newtheorem{remark}[theorem]{\protect\remarkname}

\newtheorem{example}[theorem]{\protect\examplename}

\newtheorem{assumption}[theorem]{\protect\assumptionname}
\newtheorem{algorithm}[theorem]{\protect\algorithmname}

\providecommand{\theoremname}{Theorem}
\providecommand{\propositionname}{Proposition}
\providecommand{\corollaryname}{Corollary}
\providecommand{\factname}{Fact}
\providecommand{\lemmaname}{Lemma}
\providecommand{\assumptionname}{Assumption}
\providecommand{\algorithmname}{Algorithm}
\providecommand{\definitionname}{Definition}
\providecommand{\notationname}{Notation}
\providecommand{\remarkname}{Remark}
\providecommand{\examplename}{Example}
\providecommand{\claimname}{Claim}
\providecommand{\algorithmname}{Algorithm}
\providecommand{\openprobname}{Open Problem}

\let\originalleft\left
\let\originalright\right
\renewcommand{\left}{\mathopen{}\mathclose\bgroup\originalleft}
\renewcommand{\right}{\aftergroup\egroup\originalright}

\DeclarePairedDelimiter{\norm}{\lVert}{\rVert}
\DeclarePairedDelimiterX{\scal}[2]{\langle}{\rangle}{  #1 \, \delimsize \vert \, \mathopen{}  #2  } 
\DeclarePairedDelimiterX\menge[2]{ \{ }{ \} }{ {#1} ~ \delimsize \vert ~ \mathopen{}  {#2} }  
\DeclarePairedDelimiterX\fa[2]{ ( }{ )_{#2} }{#1}  
\DeclarePairedDelimiterX\set[2]{ \{ }{ \}_{#2} }{#1}  

\DeclarePairedDelimiter\rbr{ ( }{ ) }
\DeclarePairedDelimiter\rb{ ( }{ ) }

\DeclarePairedDelimiter\sbrc{ [ }{ ] }
\newcommand{\HH}{\ensuremath{\mathcal H}}
\newcommand{\RR}{\ensuremath{\mathbb R}}
\newcommand{\NN}{\ensuremath{\mathbb N}}

\newcommand{\NPP}{\ensuremath{\mathbb{N}^{\ast}}}
\newcommand{\minf}{\ensuremath{ {-}\infty}}
\newcommand{\pinf}{\ensuremath{ {+}\infty}}
\newcommand{\RX}{\ensuremath{ \left] \minf, \pinf \right] }}

\newcommand{\RP}{\ensuremath{\mathbb{R}_{+}}}
\newcommand{\RPP}{\ensuremath{\left]0,\pinf \right[}}

\DeclareMathOperator*{\Argmin}{Argmin}

\newcommand{\closu}[1]{\ensuremath{\overline{#1} }}
\newcommand{\inte}{\ensuremath{\operatorname{int}}}

\newcommand{\ran}{\ensuremath{\operatorname{ran}}}
\newcommand{\dom}{\ensuremath{\operatorname{dom}}}

\newcommand{\Id}{\ensuremath{\operatorname{Id}}}
\newcommand{\prox}[1]{\ensuremath{\mathop{\operatorname{Prox}_{#1}}}}
\newcommand{\pr}{\ensuremath{\operatorname{P}}} 
\newcommand{\grad}[1]{\ensuremath{\mathop{\nabla {#1} }   }}

\newcommand{\ball}[2]{\ensuremath{\operatorname{B}\left({#1};{#2}\right)}}

\let\subset\subseteq
\newcommand{\email}[1]{\href{mailto:#1}{\nolinkurl{#1}}} 
\newcommand{\tsp}{\hspace{-0.15em}}
\let\oldFootnote\footnote
\newcommand\nextToken\relax
\renewcommand\footnote[1]{%
    \oldFootnote{#1}\futurelet\nextToken\isFootnote}
\newcommand\isFootnote{%
    \ifx\footnote\nextToken\textsuperscript{,}\fi}
\newcommand{\bigO}[1]{\ensuremath{ {\mathrm{O}}{#1 }}}
\newcommand\smallO[1]{
\mathchoice
{
  {\mathrm{o}}{#1}
}
{
  {\mathrm{o}}{#1}
}
{
  {\mathrm{o}}{#1}
}
{
  \scalebox{0.8}{$\mathrm{o}$}{#1}
}
}
\newcommand{\tauinf}{\ensuremath{\tau_{\infty}}}

\begin{document}

\title{\sffamily  
Applying FISTA to optimization problems (with or) without minimizers
}

\author{
Heinz H. Bauschke\thanks{
Mathematics, University of British Columbia, Kelowna, B.C.\ V1V~1V7, Canada. 
Email: \email{heinz.bauschke@ubc.ca}.},~
Minh N. Bui\thanks{
Department of Mathematics,
North Carolina State University,
Raleigh, NC 27695-8205, USA.
Email: \email{mnbui@ncsu.edu}.},~
and Xianfu Wang\thanks{
Mathematics, University of British Columbia, Kelowna, B.C.\ V1V~1V7, Canada. 
Email: \email{shawn.wang@ubc.ca}.}
}

\date{July 2, 2019}

\maketitle

\begin{abstract}
\noindent
Beck and Teboulle's FISTA method for finding a minimizer
of the sum of two convex functions, 
one of which has a Lipschitz continuous gradient
whereas 
the other may be nonsmooth, 
is arguably the most important
optimization algorithm of the past decade. While research activity
on FISTA has exploded ever since, the mathematically challenging case when the original optimization
problem has no minimizer has found only limited attention. 

In this work, we systematically study FISTA and its variants. 
We present general results that
are applicable, regardless of the existence of minimizers. 
\end{abstract}

{\small

}

\section{Introduction}
\label{sec:intro}
We assume that 
\begin{empheq}[box=\mybluebox]{equation}
\label{eq:Hilbert.space}
\text{$\HH$ is a real Hilbert space}
\end{empheq}
with inner product $\scal{\, \cdot}{\cdot\,} $
and associated norm $\norm{\, \cdot \, } $.
We also presuppose throughout the paper that 
\begin{empheq}[box=\mybluebox]{equation}
f \colon \HH \to \RR
\quad\text{and}\quad
g \colon \HH \to \RX
\end{empheq}
satisfy the following:

\begin{assumption}\label{assump:1} \ 
	\begin{enumerate}[label = \bfseries(A\arabic{enumi})]
		\item\label{assump:f} $f$
		is convex and Fr\'{e}chet
		differentiable on $\HH$,
        and $\grad{f}$ is 
		$\beta$-Lipschitz continuous
		with  $\beta \in \RPP$;
		\item\label{assump:g} $g$
		is convex,
		lower semicontinuous,
		and proper;
		\item\label{assump:gamma} $\gamma \in \left]0,1/\beta \right]$
		is a parameter.
\end{enumerate}
\end{assumption}
One fundamental problem in optimization is to 
\begin{equation}
  \label{e:theprob}
\text{minimize $f+g$ over $\HH$.}
\end{equation}
For convenience, we set 
\begin{empheq}[box = \mybluebox]{equation}
\label{eq:T.and.hINTRO}
h \coloneqq f + g \quad 
\text{and}
\quad 
T \coloneqq \prox{\gamma g}\circ\mathop{\rb[\big]{\Id - \mathop{\gamma {\grad{f}}}}},
\end{empheq}
and where we follow standard notation in convex analysis 
(as employed, e.g., in \cite{Bauschke-Combettes-2017}). 
Then many algorithms designed for solving \cref{e:theprob} employ
the forward-backward or proximal gradient operator $T$ in some fashion. 
Since the advent of Nesterov's acceleration \cite{Nesterov-1983} (when $g\equiv 0$)
and Beck and Teboulle's fast proximal gradient method FISTA
\cite{BeckTeboulle-FISTA}
(see also \cite[Chapter~10]{BeckSIAM}),
the literature on algorithms
for solving \cref{e:theprob} has literally exploded; see, e.g.,
\cite{Nesterov-1983, BeckTeboulle-FISTA, Aujol-Dossal-2015,
Attouch-Cabot-HAL2017, Attouch-fast-MPB-2018,
Attouch-rate-2016,Combettes-Glaudin,Attouch-JOTA-18} for
a selection of key contributions. 
Indeed, 
out of nearly one million mathematical publications that appeared
since 2009 and are indexed by \emph{Mathematical Reviews}, 
the 2009-FISTA paper \cite{BeckTeboulle-FISTA} by Beck and Teboulle
takes the \emph{number two spot!} (In passing, we note that it has been
cited more than 6,000 times on Google Scholar where it now receives about 
\emph{3 new citations every day!})
The overwhelming majority of these papers assume that the problem \cref{e:theprob}
has a solution to start with. 
Complementing and contributing to these analyses, we follow
a path less trodden: 

{\em 
The aim of this paper is to study the behaviour of the fast proximal gradient
methods (and monotone variants), in the case when
the original problem \cref{e:theprob} does not necessarily have a solution.
}

Before we turn to our main results, let us state the FISTA or fast proximal gradient method:
\begin{algorithm}[FISTA]
  Let $x_{0} \in \HH$, 
  set $y_{1}\coloneqq x_{0}$,
  and update 
	\begin{align}
	& \text{for~} n=1,2,\ldots \notag \\
	& \left\lfloor
	\begin{array}{ll}
	x_{n} & \coloneqq 
	Ty_{n}, \smallskip  \\ 
	y_{n+1} & \displaystyle  
	\coloneqq 
	x_{n} + \frac{\tau_{n}-1}{\tau_{n+1}}
    \rb{x_{n} - x_{n-1}},
	\end{array}
	\right.
	\end{align}
  where $T$ is
  defined in \cref{eq:T.and.hINTRO},
  $\NPP \coloneqq \{1,2,\ldots\}$, and $(\tau_n)_{n\in\NPP}$ is a sequence of
  real numbers in $\left[1,+\infty\right[$. 
\end{algorithm}

Note that when $\tau_n\equiv 1$, one obtains the classical (unaccelerated) 
proximal gradient method. 
There are two very popular choices for the sequence $(\tau_n)_{n\in\NPP}$ to
achieve acceleration. 
Firstly, given $\tau_1 \coloneqq 1$, 
the classical FISTA 
\cite{BeckTeboulle-FISTA,BeckTeboulle-MFISTA,Chambolle-Pock-2016,Nesterov-1983}
update is 
\begin{equation}
\rb{\forall n \in \NPP} \quad 
\tau_{n+1} \coloneqq \frac{1+\sqrt{1+4\tau_{n}^{2}} }{2}.
\end{equation}
The second update has the explicit formula
\begin{equation}
\rb{\forall n \in \NPP}
\quad 
\tau_{n} \coloneqq 
\frac{{n+\rho-1}}{\rho},
\end{equation}
where 
$\rho \in \left[2,\pinf\right[$; 
see, e.g.,  
\cite{Attouch-fast-MPB-2018,Attouch-rate-2016,Chambolle-Dossal-15,Su-Boyd-Candes-2016}. 

Convergence results of the sequence
generated by FISTA
under a suitable tuning
of $\fa{\tau_{n}}{n \in \NPP} $ can be found in 
\cite{Attouch-rate-2016,Attouch-Cabot-HAL2017,Chambolle-Dossal-15}.
The relaxed case was considered in \cite{Aujol-Dossal-2015}
and error-tolerant versions
were considered in \cite{Attouch-fast-MPB-2018,Attouch-JOTA-18}.
In addition, for results concerning the
rate of convergence of function values, 
see \cite{BeckTeboulle-FISTA,BeckTeboulle-MFISTA,Su-Boyd-Candes-2016,Schmidt-Nicolas-Francis-2011}.
The authors of \cite{Chambolle-Pock-2016} established 
a variant of FISTA that covers the strongly convex case.
An alternative of the classical proximal gradient algorithm
with relaxation and error is presented in 
\cite{Combettes-Salzo-Villa-2018}
(see also \cite{Bauschke-Combettes-2017,Bredies-2009,Schmidt-Nicolas-Francis-2011}).
Finally, a new forward-backward
splitting scheme (for finding
a zero of a sum of two maximally monotone operators) that
includes FISTA as a special case
was proposed in \cite{Combettes-Glaudin}.

{\em 
The main difference between our work and existing work is that
we focus on the minimizing property of the sequences generated by 
FISTA and MFISTA in the general framework, 
i.e., when the set $\Argmin\rbr{f+g} $ is possibly empty.
}
Let us now list our \textbf{main results}:
\begin{itemize}
\item \cref{t:general.case}
 establishes the behavior
 of FISTA in the possibly
 inconsistent case; 
 moreover, our assumption
 on $\fa{\tau_{n}}{n \in \NPP}$ (see \cref{e:cond-tau})
 is very mild. 
\item \cref{t:main-fista} concerns FISTA 
 when $\fa{\tau_{n}}{n \in \NPP} $ 
 behaves
 similarly to the Beck--Teboulle choice. 
\item \cref{t:main.ista} deals 
 with the case when $\fa{\tau_{n}}{n \in \NPP} $
 is bounded; see, in particular, 
 \cref{i:ista.2a} and \cref{i:ista.5ab}.
\item \cref{t:mfista}
 considers MFISTA \cite{BeckTeboulle-MFISTA}, 
 the monotone version of FISTA, when 
\cref{assump:2} is in force and 
$\fa{\tau_{n}}{n \in \NPP}$ is unbounded. 
\end{itemize}
To the best of our knowledge, \cref{t:general.case} is new.
The proof of \cref{t:main-fista}, which can be viewed as a ``discrete version''
of \cite[Theorem~2.3]{Attouch-fast-MPB-2018},
relies on techniques seen in 
\cite[Theorem~2.3]{Attouch-fast-MPB-2018}
and \cite[Proposition~3]{Attouch-Cabot-HAL2017};
items \cref{i:main1,i:main2,i:main3,i:main4,i:main5} are new. 
A result similar to \cref{t:main-fista}\cref{i:main1}
was mentioned in \cite[Theorem~4.1]{Attouch-fast-arxiv-2015}.
However, no proof was given, and 
the parameter sequence there    
is a special case of the one considered in
\cref{t:main-fista}.   
Items \cref{i:main.6a} and 
\cref{i:main.6b} is a slight modification of 
\cite[Proposition~4.3]{Attouch2017rate}.
Concerning \cref{t:main.ista}, items 
\cref{i:ista.1,i:ista.2,i:ista.3,i:ista.4,i:ista.5ab} are new
while \cref{i:ista.5a} was proven in 
\cite[Corollary~20(iii)]{Attouch-Cabot-HAL2017}.
Item~\cref{i:ista.1} in the classical case 
($\tau_{n} \equiv 1$) relates to
\cite[Theorem~4.2]{Cruz-Nghia-2016} where
linesearches were employed. 
In \cref{t:mfista}, items \cref{i:mfista.1,i:mfista.2,i:mfista.3,i:mfista.4,i:mfista.5} are new.
Compared to \cite[Theorem~5.1]{BeckTeboulle-MFISTA}, 
we allow many possible choices for 
the parameter sequence in \cref{t:mfista}\cref{i:mfista.6};
see, e.g., \cref{eg:seq,eg:Attouch-cond,eg:Aujol-seq}.
In addition, by adapting the technique
of \cite[Theorem~9]{Attouch-Cabot-HAL2017},
we improve the convergence rate of MFISTA under the condition
\cref{e:Attouch.cond2} in \cref{t:mfista}.

There are also several \textbf{minor results} worth emphasizing:
\cref{l:seq.1} is new. 
The notion of quasi-Fej\'{e}r monotonicity is revisited in \cref{l:fejer};
however, our error sequence need not be positive.
The assumptions in \cref{l:1-fista-step}
and \cref{l:MFISTA.1step} are somewhat minimal,
which allow us to establish
the minimizing property of FISTA and MFISTA
in the case where there are possibly
no minimizers in \cref{fista,mfista}. 
\cref{eg:Attouch-cond} is new. 
\cref{p:ISTA} describes the behaviour of $\fa{x_{n}-x_{n-1}}{n \in \NPP}$ in the classical 
proximal gradient (ISTA) case while 
\cref{c:ista.convergent} provides a sufficient condition for strong convergence of 
$(x_n)_{n\in\NPP}$ in this case. 
The new \cref{p:convergence} presents some progress towards the still open question regarding
the convergence of $\fa{x_{n}}{n \in \NPP} $ generated by classical FISTA.
The weak convergence part in \cref{c:ista.convergent}
was considered in \cite{Attouch2017rate};
however, our new Fej\'{e}rian approach allows us 
to obtain strong convergence when $\inte\rbr{\Argmin h} \neq \varnothing$.

Let us now turn to the organization of this paper. 
Classical results on real sequences and new results on the Fej\'{e}r
monotonicity are recorded in \cref{au}.
The ``one step'' behaviour  
of both FISTA and MFISTA is carefully examined in \cref{pre}.
In \cref{sec:parseq}, we investigate properties of
the parameter sequence $(\tau_n)_{n\in\NPP}$. 
Our main results on FISTA and MFISTA are presented in 
\cref{fista,mfista} respectively. 
The concluding \cref{openprobs} contains a discussion of open problems. 

A final note on notation is in order. 
For a sequence $\fa{\xi_{n}}{n \in \NPP} $ and 
an extended real number 
$\xi \in \left[\minf,\pinf \right]$,
the notation $\xi_{n}\uparrow \xi$ means that 
$\fa{\xi_{n}}{n\in \NPP} $ is increasing (i.e., $\xi_n\leq\xi_{n+1}$) and 
$\xi_{n} \to \xi$
as $n\to \pinf$. Likewise, $\xi_{n}\downarrow \xi$
means that $\fa{\xi_{n}}{n \in \NPP} $ is decreasing (i.e., $\xi_{n}\geq\xi_{n+1}$)
and $\xi_{n}\to \xi$ as $n\to \pinf$.
For any other notation not defined, we refer the reader to 
\cite{Bauschke-Combettes-2017}.

\section{Auxiliary results}
\label{au}

In this section, we collect results on sequences
which will make the proofs in later sections more structured.

\begin{lemma}\label{l:blowsup}
  Let $\fa{\tau_{n}}{n \in \NPP} $
  be an increasing sequence 
  in  $\left[1,\pinf \right[$
    such that $\lim  \tau_{n} = \pinf$.
    Then 
    \begin{equation}
      \label{eq:sum-blowsup}
    \sum_{n \in \NPP} \rbr*{1- \rbr[\bigg]{
    \frac{\tau_{n}-1}{\tau_{n+1}}}^{\tsp 2}  } 
    = \sum_{n \in \NPP} \rbr*{1-\frac{\tau_{n}^{2}}{\tau_{n+1}^{2}}} 
    = \pinf.
    \end{equation}
\end{lemma}
\begin{proof}
See \cref{app:blowsup}.
\end{proof}
\begin{lemma}
  \label{l:summable-liminf}
  Let $\fa{\alpha_{n}}{n \in \NPP} $ and 
  $\fa{\beta_{n}}{n \in \NPP} $ be sequences in $\RP$.
  Suppose that $\sum_{n \in \NPP}\alpha_{n} = \pinf$
  and that $\sum_{n \in \NPP}\alpha_{n}\beta_{n} < \pinf$.
  Then $\varliminf \beta_{n}  =0 $.
\end{lemma}
\begin{proof}
  See \cref{app:liminf-sum}.
\end{proof}

The novelty of the following result lies in the fact that
the error sequence $\fa{\varepsilon_{n}}{n \in \NPP}$ need not
lie in $\RP$. 

\begin{lemma}
\label{l:summable-limit}
Let $\fa{\alpha_{n}}{n \in \NPP} $
be a sequence in $\RR$,
let $\fa{\beta_{n}}{n \in \NPP} $
be a sequence in $\RP$,
and let $\fa{\varepsilon_{n}}{n \in \NPP} $
be a sequence in $\RR$.
Suppose that $\fa{\alpha_{n}}{n \in \NPP} $
is bounded below, that 
\begin{equation}
\label{eq:seqs-cond}
\rbr{\forall n \in \NPP} 
\quad 
\alpha_{n+1} 
\leq \alpha_{n} -\beta_{n} +\varepsilon_{n},
\end{equation}
and that the series 
$\sum_{n \in \NPP}\varepsilon_{n}$ converges in $\RR$.
Then the following hold:
\begin{enumerate}
\item \label{i:convergence} 
$\fa{\alpha_{n}}{n \in \NPP} $ is convergent in $\RR$.
\item \label{i:summable} $\sum_{n \in \NPP}\beta_{n} < \pinf$.
\end{enumerate}
\end{lemma}
\begin{proof}
See \cref{app:sum-limit}. 
\end{proof}
\begin{lemma}
  \label{l:seq.1}
  Let $\fa{\alpha_{n}}{n \in \NPP} $ be a
  sequence of real numbers.
  Consider the following statements:
  \begin{enumerate}    
    \item\label{i:3.1}  $\fa{n\alpha_{n}}{ n \in \NPP} $ converges in $\RR$.
    \item\label{i:3.2}  $\sum_{n \in \NPP}\alpha_{n}$ converges
      in $\RR$.
    \item\label{i:3.3} $\sum_{n \in \NPP} n \rbr{\alpha_{n}-\alpha_{n+1}} $
      converges in $\RR$.
  \end{enumerate}
  Suppose that two of
  the statements \cref{i:3.1,i:3.2,i:3.3} hold.
  Then the remaining one also holds.
\end{lemma}
\begin{proof}
  See \cref{app:seq1}.
\end{proof}

The following result is stated in 
\cite[Problem~2.6.19]{Radulescu09};
we provide a proof  in \cref{app:seq2} for completeness.

\begin{lemma}
  \label{l:seq2}
  Let $\fa{\alpha_{n}}{n \in \NPP} $ be a
  decreasing sequence in $\RP$.
  Then 
  \begin{equation}
    \sum_{n \in \NPP}\alpha_{n } < \pinf
    \quad
    \iff
      \quad 
      \Bigl[\, 
        n\alpha_{n}  \to 0
        \text{ as } n\to \pinf
        \text{ and }
        \sum_{n \in \NPP}n\rbr{\alpha_{n}-\alpha_{n+1}}  
        < \pinf
      \,\Bigr].
  \end{equation}
\end{lemma}

The following variant
of Opial's lemma 
 \cite{Opial1967weak}
 will be required
in the sequel.

\begin{lemma}
  \label{l:opial.variant}
  Let $C $ be a nonempty subset of $\HH$,
  and let 
  $\fa{u_{n}}{n \in \NPP} 
  $ and $\fa{v_{n}}{n \in \NPP} $
  be sequences in $\HH$.
  Suppose that
  $u_{n} - v_{n} \to 0$,
  that every weak sequential
  cluster point of $\fa{v_{n}}{n \in \NPP} $
  lies in $C$,
  and that,
  for every $c \in C$,
  $\fa{\norm{u_{n}-c} }{n \in \NPP} $
  converges.
  Then 
  there exists $w \in C$
  such that
  $u_{n} \rightharpoonup w$
  and $v_{n}\rightharpoonup w$.
\end{lemma}

\begin{proof}   
  For every $c \in C$,
  since $u_{n}-v_{n}\to 0$
  and $\fa{\norm{u_{n}-c} }{n \in \NPP}  $
  converges,
  we deduce that $\fa{\norm{v_{n}-c} }{n \in \NPP} $
  converges.
  In turn, because every 
  weak sequential cluster point
  of $\fa{v_{n}}{n \in \NPP} $ belongs 
  to $C$, \cite[Lemma~2.47]{Bauschke-Combettes-2017}
  yields the existence of $w \in C$
  satisfying $v_{n} \rightharpoonup w$.
  Therefore, because $u_{n}-v_{n}\to 0$,
  we conclude that $\fa{u_{n}}{n \in \NPP} $
  and $\fa{v_{n}}{n \in \NPP} $
  converge weakly to $w$.
\end{proof}

We next revisit
the notion of quasi-Fej\'{e}r
monotonicity in the Hilbert spaces setting 
studied in \cite{Combettes-fejer-2001}.
This plays a crucial role in
our analysis of \cref{p:convergence}.
Nevertheless, 
to fit our framework of \cref{p:convergence},
the error sequence $\fa{\varepsilon_{n}}{n \in \NPP} $
is not required to be positive 
in \cref{l:fejer}.
The proof is based on \cite[Proposition~3.3(iii)
and Proposition~3.10]{Combettes-fejer-2001}.

\begin{lemma}
  \label{l:fejer}
  Let $C$ be a nonempty
  subset of $\HH$,  let
  $\fa{u_{n}}{n \in \NPP} $ be a sequence in $\HH$,
  and let $\fa{\varepsilon_{n}}{n \in \NPP} $ be a
  sequence in $\RR$.
  Suppose that
  \begin{equation}
    \label{eq:fejer}
    \rbr{\forall c \in C} \rbr{\forall n \in \NPP} \quad
    \norm{u_{n+1}-c}^{2} \leq \norm{u_{n}-c}^{2} + \varepsilon_{n}, 
  \end{equation}
  and that $\sum_{n \in \NPP}\varepsilon_{n}$ converges in $\RR$.
  Then the following hold: 
  \begin{enumerate}
    \item\label{i:fejer.1} For every $c \in C$,
      the sequence $\fa{\norm{u_{n}-c}}{n \in \NPP} $
      converges in $\RR$.
    \item\label{i:fejer.2} Suppose that $\inte C \neq \varnothing$.
      Then $\fa{u_{n}}{n \in \NPP} $ converges strongly
      in $\HH$.
  \end{enumerate}
\end{lemma}

\begin{proof}
  \cref{i:fejer.1}:  
  This is a direct consequence of 
  \cref{l:summable-limit}\cref{i:convergence}.

  \cref{i:fejer.2}: 
  We follow along the lines of 
  \cite[Proposition~3.10]{Combettes-fejer-2001}.
  Let $v \in \inte C$ and $\rho \in \RPP$
  be such that $\ball{v}{\rho}\coloneqq \menge{ x \in \HH }{\norm{x-v} \leq \rho}  \subset C$.
  Define a sequence $\fa{v_{n}}{n \in \NPP} $
  in $C$ via 
  \begin{equation}
    \label{eq:vn.inte}
    \rbr{\forall n \in \NPP} \quad
    v_{n} \coloneqq
    \begin{cases}
      v, & \text{if~} u_{n+1}= u_{n}; \smallskip \\
      \displaystyle
      v -\rho \frac{u_{n+1}- u_{n}}{\norm{u_{n+1}-u_{n}} },
      &\text{otherwise}.
    \end{cases}
  \end{equation}
  We now verify that 
  \begin{equation}
    \label{eq:verify.inte}
    \rbr{\forall n \in \NPP} \quad
  \norm{u_{n+1}-v}^{2}
  \leq \norm{u_{n}-v}^{2} - 2\rho \norm{u_{n+1}- u_{n}}  
  +\varepsilon_{n}.
  \end{equation}
  Fix $n \in \NPP$.
  If $u_{n+1}=u_{n}$,
  then \cref{eq:fejer}
  implies that $\varepsilon_{n} \geq  0$,
  and therefore \cref{eq:verify.inte} holds.
  Otherwise, because $v_{n} \in C$,
  \cref{eq:fejer} yields
  $\norm{u_{n+1}-v_{n}}^{2} \leq  \norm{u_{n}-v_{n}}^{2} +\varepsilon_{n}$.
  In turn, using \cref{eq:vn.inte}, we obtain
  \begin{equation}
    \label{eq:expans.obtain}
    \norm*{\rbr{u_{n+1}-v} + \rho \frac{u_{n+1}-u_{n}}{\norm{u_{n+1}- u_{n}} }}^{2} 
    \leq 
    \norm*{\rbr{u_{n}-v} + \rho \frac{u_{n+1}-u_{n}}{\norm{u_{n+1}- u_{n}} }}^{2} 
    +\varepsilon_{n},
  \end{equation}
  and after expanding both sides
  and simplifying terms,
  we get \cref{eq:verify.inte}.
  Consequently, owing to \cref{eq:verify.inte}
  and the convergence of $\sum_{n \in \NPP}\varepsilon_{n}$,
  we derive from \cref{l:summable-limit}\cref{i:summable}
  that $\sum_{n \in \NPP}2\rho\norm{u_{n+1}-u_{n}} < \pinf$.
  Hence, by completeness of $\HH$, 
  $\fa{u_{n}}{n \in \NPP} $ converges strongly
to a point in $\HH$.
\end{proof}

We conclude this section with a simple identity. 
If $x$, $y$, and $z$ are in \HH, then
  \begin{equation}
    \label{e:points}
    \norm{x-y}^{2} + 2 \scal{x-y}{z-x} = \norm{z-y}^{2} -\norm{z-x}^{2}. 
  \end{equation}

\section{One-step results}
\label{pre}

The aim of this section is to present several results on 
performing just \emph{one step} of FISTA or MFISTA. 
This allows us to present subsequent convergence results more clearly. 
Recall that \cref{assump:1} is in force and 
(see \eqref{eq:T.and.hINTRO}) that 
\begin{equation}
      \label{eq:T.and.h}
	h = f + g \quad 
	\text{and}
	\quad 
      T = \prox{\gamma g}\circ\mathop{\rb[\big]{\Id - \mathop{\gamma {\grad{f}}}}}.
\end{equation}
Clearly, 
	\begin{equation}\label{e:important-note}
	\ran T 
      = \ran \rbr[\big]{\prox{\gamma g} \circ \mathop{\rb{\Id - \mathop{\gamma {\grad{f}}} }}}
	\subseteq \dom \partial g \subseteq \dom{g} = \dom{h}.
	\end{equation}

\begin{lemma}[Beck--Teboulle] \label{f:key-ineq}
	The following holds:
	\begin{equation}\label{e:key-ineq}
		\rb{\forall \rb{x,y} \in \HH \times \HH}
		\quad 
		\gamma^{-1}\scal{y-Ty}{x-y} + 
		\rb{2\gamma}^{-1}\norm{y-Ty}^{2}
		\leq h\rb{x} - h\rb{Ty} .
	\end{equation}
\end{lemma}

\begin{proof}
  See \cref{app:keylem}.  
\end{proof}

\begin{lemma}[one FISTA step]
	\label{l:1-fista-step}
	Let $\rb{y,x_{-}} \in \HH\times \HH$, 
	let $\tau$
	and $\tau_{+}$ be in 
    $ \left[1, \pinf \right[ $,
	and set
	\begin{equation}\label{e:one-step-dfn}
			x \coloneqq Ty,
			\quad 
			y_{+} \coloneqq x + \displaystyle 
				\frac{\tau -1}{\tau_{+}}\rb{x-x_{-}},
			\quad 
			\text{and~}
			x_{+} \coloneqq Ty_{+}.	
	\end{equation} 
	In addition, 
      let $z \in \dom{h}$,
	and 
	set 
	\begin{equation}
        \label{eq:one-step-u-mu}
	\left\{ 
		\begin{array}{ll}
			u & \coloneqq \tau x - \rb{\tau -1}x_{-} - z, \smallskip \\
			u_{+} & \coloneqq \tau_{+} x_{+} - \rb{\tau_{+} -1}x -z, \smallskip\\
			\mu & \coloneqq h\rb{x} - h\rb{z},\smallskip\\
			\mu_{+} & \coloneqq h\rb{x_{+}} - h\rb{z}. 
		\end{array}\right.
	\end{equation}
	Then the following hold:
	\begin{enumerate}
		\item\label{i:1step-1}
		\begin{math}
			h\rb{x_{+}}
			+\rb{2\gamma}^{-1}\norm{x_{+}-x}^{2}
			\leq h\rb{x} + 
			\rb{\tau-1}^{2}\rb{2\gamma}^{-1}
            \norm{x-x_{-}}^{2} / \tau^{2}_{+}.
		\end{math}
		\item\label{i:1step-2} \begin{math}
		\tau_{+}^{2}\mu_{+} + 
		\rb{2\gamma}^{-1}\norm{u_{+}}^{2}
		\leq 
		\tau_{+}\rb{\tau_{+}-1}\mu + 
		\rb{2\gamma}^{-1} \norm{u}^{2}.
		\end{math}
		\item\label{i:1step-3} Suppose that
          $\tau\leq\tau_{+}$,
          that 
		\begin{math}
          \tau_{+}\rb{\tau_{+} -1}\leq \tau^{2},
		\end{math}
		and that 
		$ \inf h > \minf$.
		Then 
		\begin{equation}
			\tau_{+}^{2}\mu_{+} + 
			\rb{2\gamma}^{-1}\norm{u_{+}}^{2}
			\leq 
			\tau^{2}\mu + \rb{2\gamma}^{-1}\norm{u}^{2}
			+ \tau_{+}\rb{ h\rb{z} - \inf{h} }.
		\end{equation}
	\end{enumerate}
\end{lemma}

\begin{proof}
	First, 
	since $z \in \dom{h}$, we get 
      from \cref{e:important-note}\&\cref{e:one-step-dfn}\&\cref{eq:one-step-u-mu}
	that 
	$\mu\in \RR$ and $\mu_{+} \in \RR$.
	Next, 
	because $x_{+} = Ty_{+}$,  
	we derive from 
	 \cref{e:key-ineq}
       (applied to $\rb{x,y_{+}}$) 
	that 
	\begin{equation}\label{e:1st-ineq}
	\mu-\mu_{+} = h\rb{x} - h\rb{x_{+}} 
		\geq 
		\gamma^{-1}\scal{y_{+} - x_{+}}{x-y_{+}}
		+ \rb{2\gamma}^{-1}\norm{y_{+}-x_{+}}^{2}.
	\end{equation}

	\cref{i:1step-1}: 
	We derive from \cref{e:1st-ineq}, \cref{e:points},
	and \cref{e:one-step-dfn} 
	that 
	\begin{subequations}
		\begin{align}
		h\rb{x} - h\rb{x_{+}}
		& \geq \rb{2\gamma}^{-1}\rb*{
			\norm{x-x_{+}}^{2}
			- \norm{y_{+} -x_{+}}^{2} - \norm{x-y_{+}}^{2}	}
		+\rb{2\gamma}^{-1}\norm{y_{+}-x_{+}}^{2} \\
		& = \rb{2\gamma}^{-1}\rb*{ \norm{x-x_{+}}^{2} 
			-\rb*{\frac{\tau -1}{\tau_{+}}}^{\tsp 2} \norm{x-x_{-}}^{2} },
		\end{align}
	\end{subequations}
    and thus, since $h\rbr{x_{+}}  \in \RR$,
    the conclusion follows.

	\cref{i:1step-2}: 
	Since $x_{+} = Ty_{+}$, 
	applying \cref{e:key-ineq}
	to $\rb{z,y_{+}}$
	gives 
		\begin{equation}\label{e:2nd-ineq}
          {-}\mu_{+} = h\rb{z} - h\rb{x_{+}}
		\geq 
		\gamma^{-1}\scal{y_{+} - x_{+}}{z-y_{+}}
		+ \rb{2\gamma}^{-1}\norm{y_{+}-x_{+}}^{2}.
		\end{equation}
	Therefore, because  
	$\tau_{+} -1 \geq 0$ by assumption, 
	it follows from \cref{e:1st-ineq} 
	and \cref{e:2nd-ineq} that 
	\begin{subequations}
		\label{e:last-ineq}
			\begin{align}
		\rb{\tau_{+} -1}\mu - \tau_{+} \mu_{+}
		& = \rb{\tau_{+} -1}\rb{\mu-\mu_{+}} + \rb{-\mu_{+}} \\
		& \geq \gamma^{-1}
		\scal{y_{+} -x_{+} }{ \rb{\tau_{+}-1}\rb{x-y_{+}} + \rb{z-y_{+}} }
		+ \rb{2\gamma}^{-1}\tau_{+} \norm{y_{+}-x_{+}}^{2} \\
		& = \gamma^{-1}
		\scal{y_{+} - x_{+} }{ \rb{\tau_{+} -1}x - \tau_{+} y_{+} +z }
		+  \rb{2\gamma}^{-1}\tau_{+} \norm{y_{+}-x_{+}}^{2}. 
		\end{align}
	\end{subequations}
	In turn, 
	on the one hand, 
	multiplying both  
	sides of \cref{e:last-ineq}
	by $\tau_{+} >0$,
	we infer from \cref{e:points} 
	(applied to $\rb{\tau_{+}y_{+}, \tau_{+} x_{+}, \rb{\tau_{+} -1}x  +z}$) 
	and the very definition of $u_{+}$
	that 
	\begin{subequations}
	\begin{align}
		\tau_{+}\rb{\tau_{+} -1}\mu - \tau_{+}^{2}\mu_{+}
		& \geq 
		\gamma^{-1}\scal{ \tau_{+} y_{+} - \tau_{+} x_{+}
}{ \rb{\tau_{+} -1}x+z - \tau_{+} y_{+}}
		+ \rb{2\gamma}^{-1}\norm{\tau_{+}\rb{y_{+} - x_{+}}}^{2} \\
		& = \rb{2\gamma}^{-1}
            \rb[\big]{\norm{ \rb{\tau_{+} -1}x  +z - \tau_{+} x_{+} }^{2} - 
			\norm{ \rb{\tau_{+} -1}x  +z  - \tau_{+} y_{+} }^{2}} \\
		& = \rb{2\gamma}^{-1}\rb[\big]{\norm{u_{+}}^{2}  -  
		\norm{ \rb{\tau_{+} -1}x  +z  - \tau_{+} y_{+} }^{2}}.
	\label{e:nearly-done}
	\end{align} 
	\end{subequations}
	On the other hand,
	since $\tau_{+} y_{+} = 
	\tau_{+} x + \rb{\tau -1}\rb{x-x_{-}}$
	due to \cref{e:one-step-dfn},
	the definition of $u$
      yields 
		\begin{equation}
			\tau_{+} y_{+} - \rb{\tau_{+} -1}x -z
			= \tau_{+} x + \rb{\tau -1}\rb{x-x_{-}} 
			- \rb{\tau_{+} -1}x -z
			= \tau x - \rb{\tau -1}x_{-} -z 
			= u.
		\end{equation}	
	Altogether,  
	\begin{math}
		\rb{2\gamma}^{-1}\rb{\norm{u_{+}}^{2}  
		 - \norm{u}^{2} } \leq 
         \tau_{+}\rb{\tau_{+}-1}\mu - \tau_{+}^{2}\mu_{+},
	\end{math}
      which implies the desired conclusion.
	
	\cref{i:1step-3}: 
	Since $\mu = h\rb{x} - h\rb{z} \geq \inf{h} - h\rb{z} > \minf$
	and, by  assumption, 
	$\tau_{+}^{2} -\tau_{+} -\tau^{2} \leq 0$,
	we deduce that 
	\begin{math}
		\rb{\tau_{+}^{2} -\tau_{+} -\tau^{2}}\mu
		\leq \rb{\tau_{+}^{2} -\tau_{+} -\tau^{2}}\rb{\inf{h} - h\rb{z}}
		= \rb{\tau^{2} + \tau_{+} - \tau_{+}^{2} }\rb{h\rb{z} - \inf{h} }.
	\end{math}
	Hence, because $0 < \tau \leq \tau_{+}$
	and $h\rb{z} - \inf{h} \geq 0$,
	it follows that 
	$	\rb{\tau_{+}^{2} -\tau_{+} -\tau^{2}}\mu 
	\leq \rb{\tau^{2} + \tau_{+} - \tau_{+}^{2} }\rb{h\rb{z} - \inf{h} }
	\leq \tau_{+} \rb{h\rb{z} - \inf {h} } .$
	Consequently, \cref{i:1step-2} 
	implies that 
		\begin{subequations}
			\begin{align}
			\tau_{+}^{2}\mu_{+} + 
			\rb{2\gamma}^{-1}\norm{u_{+}}^{2}
			& \leq \tau_{+}\rb{\tau_{+}-1}\mu + \rb{2\gamma}^{-1}\norm{u}^{2}\\
			& =
			\tau^{2}\mu +\rb{2\gamma}^{-1}\norm{u}^{2}
			+\rb[\big]{\tau_{+}^{2} -\tau_{+} - \tau^{2} }\mu \\
			& \leq \tau^{2}\mu +\rb{2\gamma}^{-1}\norm{u}^{2}
			+ \tau_{+} \rb{h\rb{z} - \inf{h}},
			\end{align}	
		\end{subequations}
	as required.
\end{proof}

The analysis of the following lemma 
follows the lines of \cite[Theorem~5.1]{BeckTeboulle-MFISTA}.

\begin{lemma}[one MFISTA step]
  \label{l:MFISTA.1step}
  Let $\rbr{y,x_{-}} \in \HH \times \HH$,
  let $\tau$ and $\tau_{+}$
  be in $\left[1,\pinf \right[$,
  and set
  \begin{equation}
    \label{eq:defn.MFISTA}
    \left\{ 
      \begin{array}{ll}
     z& \coloneqq Ty, \smallskip \\
    x & \coloneqq \begin{cases}
      x_{-}, & \text{if~} h\rbr{x_{-}} \leq h\rbr{z}; \\
      z, &\text{otherwise},
    \end{cases} \smallskip \\
    y_{+}&\coloneqq
   \displaystyle x + \frac{\tau}{\tau_{+}}\rbr{z - x}  + 
    \frac{\tau-1}{\tau_{+}}\rbr{x-x_{-}} ,\smallskip \\
      z_{+}& \coloneqq Ty_{+}, \smallskip \\
      x_{+}& \coloneqq \begin{cases}
        x, & \text{if~} h\rbr{x} \leq h\rbr{z_{+}}; \\
        z_{+}, &\text{otherwise}.
    \end{cases}    
\end{array}\right.
  \end{equation}
  Furthermore, let $w \in \dom h$, and define
  \begin{equation}
    \label{eq:MFISTA.defn2}
  \left\{
    \begin{array}{ll}
      u & \coloneqq \tau z - \rbr{\tau -1} x_{-} -w, \smallskip\\
      u_{+} & \coloneqq \tau_{+}z_{+} - \rbr{\tau_{+} -1} x -w, \smallskip \\
      \mu & \coloneqq h\rbr{x} -  h\rbr{w} , \smallskip \\
      \mu_{+} & \coloneqq h\rbr{x_{+}}  - h\rbr{w} .
    \end{array}
  \right.
  \end{equation}
  Then the following hold:
\begin{enumerate}
  \item\label{i:mfista.1stepa}
\begin{math}
  h\rbr{x_{+}} +\rbr{2\gamma}^{-1}\norm{ z_{+} - x}^{2} 
  \leq h\rbr{x} +\rbr{2\gamma}^{-1} 
  \tau^{2} \norm{ z- x_{-}}^{2}/\tau_{+}^{2}.
\end{math}
  \item \label{i:mfista.1stepb}
    \begin{math}
    \tau_{+}^{2}\mu_{+} + \rbr{2\gamma}^{-1}\norm{u_{+}}^{2}
    \leq 
    \tau_{+}\rbr{\tau_{+}-1}\mu + \rbr{2\gamma}^{-1}\norm{u}^{2}  . 
    \end{math}
\end{enumerate}
\end{lemma}

\begin{proof}
  First, 
  since $z_{+} = Ty_{+}$, 
  using \cref{e:key-ineq}
  with $\rbr{x,y_{+}} $
  and \cref{e:points}
  with $\rbr{y_{+},z_{+},x} $
  yields
\begin{subequations}
  \label{eq:mfista.12ineq}
  \begin{align}
  \label{eq:mfista-1ineq}
  h\rbr{x}  - h\rbr{z_{+}} 
  = h\rbr{x} - h\rbr{Ty_{+}} 
  & \geq \gamma^{-1} \scal{y_{+} - z_{+}}{x-y_{+}} + 
  \rbr{2\gamma}^{-1}\norm{y_{+} - z_{+}}^{2}.
  \\
  & = \rbr{2\gamma}^{-1}\rbr*{ 
  \norm{x-z_{+}}^{2} - \norm{x-y_{+}}^{2}}.\label{eq:mfista.2ineq}
\end{align}
\end{subequations}

  \cref{i:mfista.1stepa}:
  On the one hand,
  by the very definition of $x_{+}$
  and \cref{eq:mfista.12ineq}, 
\begin{math}
  h\rbr{x} - h\rbr{x_{+}} 
  \geq  h\rbr{x}  - h\rbr{z_{+}} 
  \geq \rbr{2\gamma}^{-1}\rbr{\norm{x-z_{+}} - \norm{x-y_{+}}^{2} }, 
\end{math}
and thus, 
\begin{equation}
  \label{eq:mfista.3ineq}
  h\rbr{x_{+}}  + 
\rbr{2\gamma}^{-1}\norm{x-z_{+}}^{2} 
  \leq h\rbr{x} +\rbr{2\gamma}^{-1}\norm{x-y_{+}}^{2}. 
\end{equation}
On the other hand, 
due to \cref{eq:defn.MFISTA},
\begin{subequations}
  \begin{align}
  \label{eq:mfista.yx}
  y_{+} - x
  & = \frac{\tau}{\tau_{+}}\rbr{z-x} + \frac{\tau-1}{\tau_{+}}\rbr{x-x_{-}} \\
  & = 
\begin{cases}
\displaystyle \frac{\tau}{\tau_{+}}\rbr{z-x_{-}}  + 
  \frac{\tau-1}{\tau_{+}}\rbr{x_{-} - x_{-}} , 
  &\text{if~} h\rbr{x_{-}} \leq h\rbr{z} ; \medskip \\
  \displaystyle 
  \frac{\tau}{\tau_{+}}\rbr{z - z}  + 
  \frac{\tau-1}{\tau_{+}}\rbr{z-x_{-}}, &\text{otherwise}
\end{cases} \\
  & = \begin{cases}
     \displaystyle 
     \frac{\tau}{\tau_{+}}\rbr{z-x_{-}}, & \text{if~} h\rbr{x_{-}} \leq  h\rbr{z};
     \medskip \\
 \displaystyle
     \frac{\tau-1}{\tau_{+}}\rbr{z-x_{-}} , &\text{otherwise}, 
  \end{cases}
\end{align}
\end{subequations}
and since $\tau\geq 1$, it follows that 
\begin{equation}
  \label{eq:mfista.4ineq}
  \norm{y_{+}-x} \leq \frac{\tau}{\tau_{+}}\norm{z-x_{-}}.
\end{equation}
Altogether, \cref{eq:mfista.3ineq} and \cref{eq:mfista.4ineq}
yield the desired result.

\cref{i:mfista.1stepb}: 
Applying \cref{e:key-ineq}
to the pair $\rbr{w,y_{+}} $
and noticing that $z_{+} = Ty_{+}$, we get
\begin{equation}
  \label{eq:mfista.5ineq}
  h\rbr{w}  - h\rbr{z_{+}} 
  \geq \gamma^{-1}\scal{y_{+} - z_{+}}{w-y_{+}} 
  +\rbr{2\gamma}^{-1}\norm{y_{+} - z_{+}}^{2}.
\end{equation}
In turn, since $\tau_{+}\geq 1$, 
the very definition of $x_{+}$, 
\cref{eq:mfista.12ineq}, and 
\cref{eq:mfista.5ineq}
imply that 
\begin{subequations}
  \begin{align}
    \rbr{\tau_{+}-1} \mu - \tau_{+}\mu_{+}
    & = \rbr{\tau_{+}-1} \rbr*{h\rbr{x} -h\rbr{w} } 
    -\tau_{+} \rbr*{h\rbr{x_{+}}  - h\rbr{w} } \\
    & = \rbr{\tau_{+}-1} h\rbr{x} + h\rbr{w} -\tau_{+}h\rbr{x_{+}} \\
    & \geq  \rbr{\tau_{+} -1} h\rbr{x} +h\rbr{w} - \tau_{+}h\rbr{z_{+}} \\
    & = \rbr{\tau_{+}-1} \rbr{h\rbr{x} -h\rbr{z_{+}}  } 
    + h\rbr{w} -h\rbr{z_{+}} \\ 
    & \geq  \rbr{2\gamma}^{-1}\tau_{+}\norm{y_{+}-z_{+}}^{2}
    +\gamma^{-1}\scal{y_{+}- z_{+}}{ \rbr{\tau_{+}-1} \rbr{x-y_{+}} +w-y_{+} } 
    \\
    & =   \rbr{2\gamma}^{-1}\tau_{+}\norm{y_{+}-z_{+}}^{2}
    +\gamma^{-1}\scal{y_{+}-z_{+}}{w+\rbr{\tau_{+}-1} x-\tau_{+}y_{+}}. 
  \end{align} 
\end{subequations}
Thus, 
since $\tau_{+} >0$,
it follows 
from \cref{e:points} 
(applied to $\rbr{\tau_{+}y_{+},\tau_{+}z_{+},w+\rbr{\tau_{+}-1} x} $)
that 
  \begin{subequations}
    \label{eq:mfista.6abineq}
    \begin{align}
      \tau_{+}\rbr{\tau_{+}-1} \mu -\tau_{+}^{2}\mu_{+} 
      & \geq  \rbr{2\gamma}^{-1}\tau_{+}^{2}\norm{y_{+}-z_{+}}^{2}
      +\gamma^{-1}\scal{\tau_{+}y_{+}-\tau_{+}z_{+} }{ w + \rbr{\tau_{+}-1} x
      -\tau_{+}y_{+}} \\
      & = \rbr{2\gamma}^{-1}\rbr[\big]{
        \norm{\tau_{+}z_{+}-\rbr{\tau_{+}-1} x -w}^{2}
        -\norm{\tau_{+}y_{+} -\rbr{\tau_{+}-1} x-w}^{2}
      } \label{eq:mfista.6ineq}.
    \end{align} 
  \end{subequations}
  Furthermore, by 
  the definition of $y_{+}$,
  we have 
  \begin{math}
    \tau_{+}y_{+} = 
    \tau_{+}x + \tau\rbr{z-x}  +\rbr{\tau -1} \rbr{x-x_{-}} 
    = \rbr{\tau_{+}-1} x + \tau z - \rbr{\tau-1} x_{-},
  \end{math}
  which asserts that  
  $\tau_{+}y_{+} - \rbr{\tau_{+}-1} x = \tau z - \rbr{\tau-1} x_{-}$.
  Combining this and \cref{eq:mfista.6abineq}
  entails that 
  \begin{equation}
    \tau_{+}\rbr{\tau_{+}-1} \mu
    -\tau_{+}^{2}\mu_{+}
    \geq  \rbr{2\gamma}^{-1}
    \rbr[\big]{ \norm{\tau_{+}z_{+} - \rbr{\tau_{+}-1}x -w }^{2} 
    -\norm{\tau z - \rbr{\tau-1} x_{-}-w}^{2}},
  \end{equation}
  which completes the proof.
\end{proof}

\section{The parameter sequence}

\label{sec:parseq}

A central ingredient of FISTA and MFISTA is the parameter sequence
$\fa{\tau_{n}}{n \in \NPP}$. In this section, we present
various properties of the parameter sequence as well as examples. 
From this point onwards, we will assume the following:
\begin{assumption}\label{assump:2}
	We assume 
      that 
	$\fa{\tau_{n}}{n \in \NPP}$
	is 
	a sequence
    of real numbers 
    such that
	\begin{empheq}[ box = \mybluebox]{equation}\label{e:cond-tau}
      \tau_{1} \in \left[1 ,\pinf\right[,
          \quad
      \rb{\forall n \in \NPP}
      ~
        \tau_{n+1} \in \sbrc*{\tau_{n},
        \frac{1+\sqrt{1+4\tau_{n}^{2}}}{2}}, 
          \quad
          \text{and}
          \quad 
          \tau_{\infty} \coloneqq \sup_{k\in\NPP} \tau_k. 
	\end{empheq} 
\end{assumption}

\begin{remark}
A  few observations regarding \cref{assump:2} are in order. 
  \label{rm:cond-tau}
  \begin{enumerate}
    \item It is clear from
      \cref{e:cond-tau}
      that 
      \begin{equation}
        \label{eq:taun.geq1}
        \rbr{\forall n \in \NPP} 
        \quad \tau_{n} \geq  1.
      \end{equation}
    \item 
  Because $\fa{\tau_{n}}{n \in \NPP}$
  is increasing,
  \begin{equation}
    \label{eq:lim.tau}
    \tau_{n} \uparrow \tau_{\infty} \overset{\cref{eq:taun.geq1}}{\in}
    \left[1, \pinf \right].
  \end{equation}
    \item     
  Due to 
 \cref{eq:taun.geq1}
 and 
  the assumption that $\rbr{\forall n \in \NPP} ~ \tau_{n+1}\leq 
  \rbr[\big]{1+\sqrt{1+4\tau_{n}^{2}}}/2 $,
  it is straightforward
  to verify that 
  \begin{equation}
    \label{eq:rewrite-cont-tau}
    \rbr{\forall n \in \NPP} 
    \quad
    \tau_{n+1}^{2} - \tau_{n+1} \leq \tau_{n}^{2}.
  \end{equation}
\item\label{i:difference} 
  For every $ n \in \NPP$,
  since $\tau_{n}\leq \tau_{n+1}\leq \rbr{1+\sqrt{1+4\tau_{n}^{2}}}/2
  $
  by \cref{e:cond-tau},
  it follows from \cref{eq:rewrite-cont-tau}
  and \cref{eq:taun.geq1}
  that 
  \begin{equation}
    \label{eq:difference}
    \tau_{n+1} - \tau_{n}
    = \frac{\tau_{n+1}^{2}- \tau_{n}^{2}}{\tau_{n+1}+\tau_{n}}
    \leq \frac{\tau_{n+1}}{\tau_{n+1}+\tau_{n}}
    \leq \frac{1+\sqrt{1+4\tau_{n}^{2}}}{2\rbr{\tau_{n}+\tau_{n}} }
    \leq \frac{\tau_{n}+\sqrt{\tau_{n}^{2}+4\tau_{n}^{2}}}{4\tau_{n}}
    = \frac{1+\sqrt{5}}{4} < 0.81.
  \end{equation}
  \end{enumerate}
\end{remark}

\begin{lemma}
  \label{l:quotient}
The following hold: 
  \begin{enumerate}
\item
\label{i:quot.3}
$\varlimsup(\tau_n/n)\leq\tau_1/2$.
    \item \label{i:quot.1}
      Using the convention 
      that 
      $\tfrac{1}{\pinf}= 0$, we have 
      \begin{equation}
        \label{eq:liminf.limsup}
        \frac{1-1/\tau_{\infty}}{1+1/\tau_{\infty}} -
        \frac{1}{\tau_{\infty}\rbr{\tauinf+1} }
        \leq \varliminf \frac{\tau_{n}-1}{\tau_{n+1}}
        \leq \varlimsup \frac{\tau_{n}-1}{\tau_{n+1}}
        \leq 1 -\frac{1}{\tauinf}.
      \end{equation}
    \item\label{i:quot.lim} 
      Suppose that $\lim \tau_{n} =\pinf$.
  Then 
  \begin{equation}
    \label{eq:quotient.lim}
    \lim \frac{\tau_{n}-1}{\tau_{n+1}} = 1.
  \end{equation}
  \end{enumerate}
\end{lemma}

\begin{proof}
\cref{i:quot.3}:
We claim that $\rbr{\forall n\in\NPP}$
$\tau_n\leq\tau_1(n+\sqrt{n})/2$.
The inequality is clear when $n=1$.
Assume that, for some integer $n\geq 1$,
we have $\tau_n\leq\tau_1(n+\sqrt{n})/2$.
Then, on the one hand,
we derive from \cref{e:cond-tau} that
\begin{equation}
\tau_{n+1}
\leq\frac{1+\sqrt{1+4\tau_n^2}}{2}
\leq\frac{\tau_1+\sqrt{\tau_1^2+\tau_1^2(n+\sqrt{n})^2}}{2}
=\frac{\tau_1\big(1+\sqrt{1+(n+\sqrt{n})^2}\big)}{2}.
\end{equation}
On the other hand,
since $(n+\sqrt{n+1})^2-(1+(n+\sqrt{n})^2)
=2n(\sqrt{n+1}-\sqrt{n})>0$,
we obtain $\sqrt{1+(n+\sqrt{n})^2}<n+\sqrt{n+1}$.
Altogether,
$\tau_{n+1}<\tau_1(n+1+\sqrt{n+1})/2$,
which concludes the induction argument.
Consequently, $\varlimsup(\tau_n/n)
\leq\lim\tau_1(n+\sqrt{n})/(2n)=\tau_1/2$.

  \cref{i:quot.1}: 
  First, since
  \begin{math}
    \rbr{\forall n \in \NPP} ~
    \rbr{\tau_{n}-1} /\tau_{n+1}
    \leq  \rbr{\tau_{n+1}-1} /\tau_{n+1}
    = 1-1/\tau_{n+1}
  \end{math}
  by \cref{e:cond-tau},
  we infer from
  \cref{eq:lim.tau} that 
  \begin{math}
    \varlimsup \rbr{\tau_{n}-1}/\tau_{n+1}
    \leq 1- 1/\tauinf.
  \end{math}
  Next, by \cref{eq:rewrite-cont-tau}
  and \cref{e:cond-tau}, we have 
  \begin{subequations}
    \begin{align}
      \rbr{\forall n \in \NPP} \quad
      \frac{\tau_{n}-1}{\tau_{n+1}} 
      =  \frac{\tau_{n}^{2}-1}{\tau_{n+1}\rbr{\tau_{n}+1} }
      & \geq \frac{\tau_{n+1}^{2}-\tau_{n+1}-1}{\tau_{n+1}\rbr{\tau_{n}+1}
      } \\
      & = \frac{\tau_{n+1}-1}{\tau_{n}+1} - 
      \frac{1}{\tau_{n+1}\rbr{\tau_{n}+1} } \\
      & \geq \frac{\tau_{n}-1}{\tau_{n}+1}
      -\frac{1}{\tau_{n+1}\rbr{\tau_{n}+1} } \\
      & = \frac{1-1/\tau_{n}}{1+1/\tau_{n}} 
      -\frac{1}{\tau_{n+1}\rbr{\tau_{n}+1} },
    \end{align} 
    and hence, we get from \cref{eq:lim.tau} that 
    \begin{math}
      \varliminf \rbr{\tau_{n}-1} /\tau_{n+1}
      \geq  \rbr{1-1/\tauinf}/\rbr{1+1/\tauinf} 
      - 1/\rbr{\tauinf\rbr{\tauinf+1} },
    \end{math}
    as desired.
  \end{subequations}
  
  \cref{i:quot.lim}: Follows from 
  \cref{i:quot.1} and \cref{eq:lim.tau}.
\end{proof}

\begin{example}\label{eg:seq}
  The condition
  \begin{equation}
    \label{e:181017a}
\sup_{n\in\NPP} \big(n/\tau_n\big) <+\infty
  \end{equation}
  and the quotient
  \begin{equation}
  \frac{\tau_n-1}{\tau_{n+1}}
  \end{equation}
  play significant roles in subsequent convergence results.
  Here are the two popular examples of sequences 
	that satisfy \cref{assump:2}
    as well as \eqref{e:181017a} already seen in 
    \cref{sec:intro}: 
	\begin{enumerate}
		\item\label{i:seq2} 
          \cite{BeckTeboulle-MFISTA,BeckTeboulle-FISTA,Chambolle-Pock-2016,Nesterov-1983}
		Set $\tau_{1} \coloneqq 1$, and set 
		\begin{math}
		\rb{\forall n \in \NPP}~
		\tau_{n+1} \coloneqq \rb[\big]{1+\sqrt{1+4\tau_{n}^{2}} }/2.
		\end{math}
		Then, it is straightforward to verify that 
		\begin{math}
		\rb{\forall n \in \NPP }~
		\tau_{n}^{2} - \tau_{n+1}^{2} + \tau_{n+1} =0
		\end{math}
		and that $\fa{\tau_{n}}{n \in \NPP}$
		is an
        increasing sequence in
        $\left[1,\pinf \right[$.
		Moreover, an inductive argument shows that 
		\begin{math}
		\rb{\forall n \in \NPP}~
		\tau_{n} \geq \rb{n+1}/2,
		\end{math}
		from which we obtain $\tau_\infty=+\infty$ and 
        $\sup_{n \in \NPP}\rbr{n/\tau_{n}} \leq 2$.
This and \cref{l:quotient}\cref{i:quot.3}
guarantee that $\lim(\tau_n/n)=1/2$.
Furthermore, it is part of the folklore that
\begin{equation}
\label{e:folklore}
\frac{\tau_n-1}{\tau_{n+1}}
=1-\frac{3}{n}
+\smallO{\bigg(\frac{1}{n}\bigg)};
\end{equation}
for completeness, 
a proof is provided in \cref{app:folklore}.

		\item\label{i:seq1} 
              \cite{Attouch-fast-MPB-2018,Attouch-rate-2016,Chambolle-Dossal-15,Su-Boyd-Candes-2016}
        Let $\rho \in \left[2,\pinf\right[$,
		and define 
		\begin{math}
		\rb{\forall n \in \NPP}
		~
		\tau_{n} \coloneqq 
		\rb{n+\rho-1}/\rho.
		\end{math}
		Then, 
		clearly $\fa{\tau_{n}}{n \in \NPP}$
		is 
        an increasing 
        sequence in $\left[1,\pinf \right[$
        with $\tau_\infty=+\infty$
        and, 
		for every $n \in \NPP$,
		we have $n / \tau_{n} = n\rho /\rb{n+\rho-1}
\leq \rho $,
		\begin{equation}
		\tau_{n}^{2} - \tau_{n+1}^{2} + \tau_{n+1}
		= \rb*{\frac{n+\rho-1}{\rho}}^{2}
		-	\rb*{\frac{n+\rho}{\rho}}^{2}
		+ \frac{n+\rho}{\rho}
		= \frac{\rb{\rho-2}n  +\rb{\rho-1}^{2} }{\rho^2} 
        \geq \frac{1}{4},
		\end{equation}
and
\begin{equation}
\frac{\tau_n-1}{\tau_{n+1}}
=\frac{n-1}{n+\rho}
=1-\frac{1+\rho}{n}+\bigO{\bigg(\frac{1}{n^2}\bigg)}.
\end{equation}
    \end{enumerate}
\end{example}

We now turn to examples of the condition
\begin{equation}
\label{eq:Attouch.cond.early}
\rbr{\exists \delta \in \left]0,1 \right[}\rbr{\forall n \in \NPP} \quad 
\tau_{n+1}^{2}-\tau_{n}^{2} \leq \delta \tau_{n+1},
\end{equation}
which is of some interest in \cref{fista} (see \cref{eq:Attouch.cond})
and \cref{mfista}. 
Further examples of
  sequences that satisfy \cref{eq:Attouch.cond.early}
  can be found in \cite[Section~5]{Attouch-Cabot-HAL2017}.

\begin{example}
  \label{eg:Attouch-cond}
Let $\rho \in \bigl]1, {+}\infty \bigr[ $
  and
  set 
  \begin{equation}
    \label{eq:mu-n}
    \rbr{\forall n\in \NN^{\ast}} \quad 
    \mu_{n} \coloneqq
    \frac{\tau_{n} + \rho -1}{\rho}.
  \end{equation}
  Then
  \begin{equation}
    \label{eq:ineq-mu}
    \rbr{\forall n \in \NN^{\ast}}
    \quad 
    \mu_{n+1}^{2} - 
    \mu_{n}^{2}
    \leq \frac{1+\sqrt{5}}{2\rho} \mu_{n+1}.
  \end{equation}
If $\rho > \rbr{1+\sqrt{5}}/2$, 
then the sequence $\fa{\mu_{n}}{n \in \NN^{\ast}}$
  satisfies \cref{eq:Attouch.cond.early}
with $\delta = \rbr{1+\sqrt{5}}/\rbr{2\rho} \in \left]0,1
\right[$.
\end{example}
\begin{proof}
  Indeed, since $\rbr{1+\sqrt{5}}/2 >1$,
  we derive from 
  \cref{eq:mu-n}, \cref{eq:rewrite-cont-tau},
  and \cref{eq:difference}
  that 
  \begin{subequations}
    \begin{align}
      \rbr{\forall n \in \NN^{\ast}}
      \quad \mu_{n+1}^{2} - \mu_{n}^{2} 
      &= \frac{\tau_{n+1}^{2} - \tau_{n}^{2} + 2\rbr{\rho-1}\rbr{\tau_{n+1}-\tau_{n}}}{\rho^{2}} 
      \leq \frac{\tau_{n+1} + \tfrac{1+\sqrt{5}}{2}\rbr{\rho-1} }{\rho^{2}} 
      \\
      &<
      \frac{\tfrac{1+\sqrt{5}}{2}\tau_{n+1} + \tfrac{1+\sqrt{5}}{2}\rbr{\rho-1} }{\rho^{2}} 
      =\frac{1+\sqrt{5}}{2\rho}\mu_{n+1},
    \end{align} 
  \end{subequations}
  as claimed. The remaining implication follows readily. 
\end{proof}

\begin{example} \cite{Aujol-Dossal-2015}
  \label{eg:Aujol-seq}
Let $\rbr{a,d} \in \RPP \times \RP $, set
    \begin{equation}
      \label{eq:Aujol-seq}
      \rbr{\forall n \in \NN^{\ast}} \quad 
      \tau_{n} \coloneqq \rbr[\Big]{\frac{n+a-1}{a}}^{d},
    \end{equation}
and suppose that one of the following holds:
\begin{enumerate}
\item $d=0$.
\item \label{i:aujopaujo} $d \in \left]0,1\right]$
and $a>\max\bigl\{ 1, \rbr{2d}^{1/d}  \bigr\}$.
\end{enumerate}
Aujol and Dossal's \cite[Lemma~3.2]{Aujol-Dossal-2015} yields
\begin{equation}
  \label{eq:Aujol-ineq}
\rbr{\forall n \in \NN^{\ast}} \quad 
\frac{1}{a^{d}} - \frac{2d}{a^{2d}} > 0
\quad \text{and} \quad 
\tau_{n}^{2} - \tau_{n+1}^{2} + \tau_{n+1}
\geq \rbr[\Big]{\frac{1}{a^{d}} - \frac{2d}{a^{2d}}}\rbr{n+a}^{d}
> 0.
\end{equation}
Let us add to their analysis by pointing out that 
if \cref{i:aujopaujo} holds, then \cref{eq:Attouch.cond.early} holds
with $\delta = {\rbr{2d}/a^{d}} \in \left]0,1\right[$.
Indeed, \cref{eq:Aujol-seq} and \cref{eq:Aujol-ineq} assert
    that
      \begin{equation}
      \rbr{\forall n \in \NPP} \quad 
      \tau_{n+1}^{2} - \tau_{n}^{2}
      \leq \tau_{n+1} - \rbr[\Big]{\frac{1}{a^{d}}- \frac{2d}{a^{2d}}}\rbr{n+a}^{d}
      = \tau_{n+1} - 
      \rbr[\Big]{\frac{1}{a^{d}}- \frac{2d}{a^{2d}}}a^{d}\tau_{n+1}
  = \delta\tau_{n+1}.
\end{equation}
Also, note that if $d\in\left]0,1\right[$,
then $\sup_{n \in \NN^{\ast}} \rbr{n/\tau_{n}} = \pinf$ 
(by L'H\^{o}pital's rule)
in contrast to \cref{eg:seq}.
\end{example}

\section{FISTA}
\label{fista}

In this section, we present three main results on FISTA. 
We again recall that \cref{assump:1} is in force and 
(see \eqref{eq:T.and.hINTRO}) that 
\begin{equation}
       \label{eq:T.and.h.fista}
	h = f + g \quad 
	\text{and}
	\quad 
      T = \prox{\gamma g}\circ\mathop{\rb[\big]{\Id - \mathop{\gamma {\grad{f}}}}}.
\end{equation}

\begin{algorithm}[FISTA]
  \label{alg:FISTA}
  Let $x_{0} \in \HH$, 
  set $y_{1}\coloneqq x_{0}$,
  and update 
	\begin{align}\label{e:FISTA-algo}
	& \text{for~} n=1,2,\ldots \notag \\
	& \left\lfloor
	\begin{array}{ll}
	x_{n} & \coloneqq 
	Ty_{n}, \smallskip  \\ 
	y_{n+1} & \displaystyle  
	\coloneqq 
	x_{n} + \frac{\tau_{n}-1}{\tau_{n+1}}
    \rb{x_{n} - x_{n-1}},
	\end{array}
	\right.
	\end{align}
  where $T$ is as in \cref{eq:T.and.h.fista} and $\fa{\tau_{n}}{n \in \NPP}$ satisfies
  \cref{e:cond-tau}. 
\end{algorithm}

We assume for the remainder of this section that
\begin{empheq}[box=\mybluebox]{equation}
\fa{x_{n}}{n \in \NPP} 
\text{~is a sequence generated by \cref{alg:FISTA}.}
\end{empheq}
We also set
\begin{empheq}[box=\mybluebox]{equation}
\label{e:defn-en}
  \rbr{\forall n \in \NPP} \quad
  \sigma_{n}\coloneqq h\rbr{x_{n}} + \frac{1}{2\gamma}\norm{x_{n}-x_{n-1}}^{2}
  \quad
  \text{and}
  \quad
  \alpha_{n}\coloneqq \frac{\tau_{n}-1}{\tau_{n+1}}.
\end{empheq}
Note that, by \cref{eq:taun.geq1}
and \cref{e:cond-tau},
\begin{equation}
  \label{eq:quotient.01}
  \rbr{\forall n \in \NPP} \quad
  0\leq \alpha_{n} 
  =\frac{\tau_{n}-1}{\tau_{n+1}}
  \leq \frac{\tau_{n}-1}{\tau_{n}}
  < 1.
\end{equation}

The first two items of the following result are due to Attouch and Cabot; see 
\cite[Proposition~3]{Attouch-Cabot-HAL2017}.

\begin{lemma}
  \label{l:decreasing}
  The following holds: 
  \begin{enumerate}
    \item\label{i:decrease.i} 
      \begin{math}
        \rbr{\forall n\in \NPP} ~
        \rbr{2\gamma}^{-1}\rbr{1-\alpha_{n}^{2}} \norm{x_{n}-x_{n-1}}^{2}
        \leq \sigma_{n}-\sigma_{n+1}.
      \end{math}
    \item\label{i:decrease.ii}
  The sequence $\fa{\sigma_{n}}{n \in \NPP} $
  is decreasing and convergent to
  a point in $\left[\minf,\pinf \right[$.
    \item\label{i:sum.decrease}
      Suppose that $\inf_{n \in \NPP} \sigma_{n} > \minf$.
      Then the following hold: 
      \begin{enumerate}
        \item \label{i:sum.2a}
      $\sum_{n \in \NPP}\rbr{1-\alpha_{n}^{2}} \norm{x_{n}-x_{n-1}}^{2} < 
      \pinf$.
    \item\label{i:sum.2b} Suppose that $\sup_{n \in \NPP}\tau_{n} < \pinf$.
      Then 
      $\inf_{n \in \NPP}\rbr{1-\alpha_{n}^{2}} > 0$
      and 
      $\sum_{n \in \NPP} \norm{x_{n}-x_{n-1}}^{2} < \pinf$.
      \end{enumerate}
  \end{enumerate}
\end{lemma}

\begin{proof}
  
\cref{i:decrease.i}: For every $n \in \NPP$,
\cref{l:1-fista-step}\cref{i:1step-1}
(applied to $\rb{y,x_{-},\tau,\tau_{+}} 
=   \rb{y_{n} ,x_{n-1},\tau_{n},\tau_{n+1} } $)
asserts that 
\begin{math}
\sigma_{n+1} 
\leq h\rb{x_{n}}  + \alpha_{n}^{2}\rb{2\gamma}^{-1}\norm{x_{n}-x_{n-1}}^{2} 
= \sigma_{n} - \rb{1-\alpha^{2}_{n} }\rb{2\gamma}^{-1}\norm{x_{n}-x_{n-1}}^{2},
\end{math} 
from which the desired inequality
follows.
\cref{i:decrease.ii}: A consequence of \cref{i:decrease.i}
and \cref{eq:quotient.01}.

\cref{i:sum.2a}: 
By \cref{i:decrease.i} and \cref{e:defn-en}, 
\begin{equation}
\rb{\forall n \in \NPP}
\quad 
\sum_{k=1}^{n}\frac{1-\alpha_{k}^{2}}{2\gamma}\norm{x_{k}-x_{k-1}}^{2}
\leq \sum_{k=1}^{n}\rb{\sigma_{k} - \sigma_{k+1} }
= \sigma_{1} -\sigma_{n+1} 
\leq \sigma_{1} - \inf_{k \in \NPP}\sigma_{k} < \pinf.
\end{equation}
Thus, $\sum_{n \in \NPP}\rb{1-\alpha_{n}^{2}}\norm{x_{n}-x_{n-1}}^{2} < \pinf$,
as claimed.

\cref{i:sum.2b}: 
  Because the function 
  $\RPP\to\RR\colon \xi \mapsto \rbr{\xi -1}/\xi $
  is increasing and 
  $\rbr{\forall n \in \NPP} ~
  0<\tau_{n}\leq \tau_{n+1} \leq \tauinf$,
  we see that  
  $\rbr{\forall n\in \NPP} ~ \alpha_{n}= 
  \rbr{\tau_{n}-1}/ \tau_{n+1} \leq \rbr{\tau_{n}-1}/\tau_{n} 
  \leq \rbr{\tauinf -1}/\tauinf \in \left[0,1 \right[ $;
      therefore,
      \begin{equation}
        \label{eq:summable.bound}
        \rbr{\forall n \in \NPP} \quad
        1-\alpha_{n}^{2} 
        \geq 1 - \rbr[\bigg]{\frac{\tauinf-1}{\tauinf}}^{\tsp 2} > 0.
      \end{equation}
  Combining \cref{eq:summable.bound}
  and \cref{i:sum.2a} yields the conclusion.
\end{proof}

We are ready for our first main result which
establishes a 
minimizing property of
the sequence $\fa{x_{n}}{n \in \NPP} $
generated by \cref{alg:FISTA}
in the general setting.

\begin{theorem}
  \label{t:general.case}
  The following holds:  
  \begin{equation}
    \label{eq:general.case.liminf}
    \rbr{\forall m \in \NPP} \quad
    \inf_{n\geq m }h\rbr{x_{n}} 
    =
    \lim_{n} \min_{1\leq k\leq n} 
    h\rbr{x_{k}}  
    = \varliminf_{n} h\rbr{x_{n}} 
    = \inf h.
  \end{equation}
\end{theorem}

\begin{proof}
  Let us first establish that 
  \begin{equation}
    \label{eq:prop.establish}
  \rbr{\forall m \in \NPP} \quad 
  \inf_{n \geq m}h\rbr{x_{n}} = \inf h.
  \end{equation}
  To do so, we proceed by contradiction: 
  assume that there exists $N\in \NPP$
  such that $\inf_{n \geq N} h\rbr{x_{n}} > \inf h$.
  Then, there exists $z \in \dom h$
  satisfying
  \begin{equation}
    \label{eq:z-existence}
    \minf <  h\rbr{z} < \inf_{n \geq N} h\rbr{x_{n}} .
  \end{equation}
  In turn,
  set 
  $\rbr{\forall n \in \NPP} ~
  \mu_{n} \coloneqq h\rbr{x_{n}} - h\rbr{z} 
  \text{~and~}
  u_{n} \coloneqq \tau_{n}x_{n} - \rbr{\tau_{n}-1} x_{n-1}-z$.
  For every $n \geq N$, 
  in the light of  
  \cref{l:1-fista-step}\cref{i:1step-2}
  (applied to $\rbr{y,x_{-},\tau,\tau_{+}}  = 
  \rbr{y_{n},x_{n-1},\tau_{n},\tau_{n+1}} $),
  we get 
  \begin{subequations}
  \label{eq:induction.general}
    \begin{align}
      \tau_{n+1}^{2}\mu_{n+1}
    +\rbr{2\gamma}^{-1}\norm{u_{n+1}}^{2}
    & \leq \tau_{n+1}\rbr{\tau_{n+1}-1}\mu_{n}
       + \rbr{2\gamma}^{-1} \norm{u_{n}}^{2} \\
       & = \tau_{n}^{2}\mu_{n} 
           + \rbr{2\gamma}^{-1} \norm{u_{n}}^{2} 
           - \rbr[\big]{\tau_{n}^{2}- \tau_{n+1}^{2} + \tau_{n+1}} \mu_{n}.
    \end{align}
  \end{subequations}
  Furthermore, due to \cref{eq:z-existence},
  \begin{equation}
    \label{eq:positive}
    \rbr{\forall n \geq N} \quad
    \mu_{n} = h\rbr{x_{n}}  - h\rbr{z} > 0.
  \end{equation}
  Let us consider the following
  two possible cases.
  
  (a) $\tau_{\infty} = \pinf$:
  By \cref{eq:lim.tau}, 
  $\tau_{n}\to \pinf$.
  Next, we derive from \cref{eq:induction.general}, 
  \cref{eq:rewrite-cont-tau},
  and \cref{eq:positive} that
  \begin{math}
    \rbr{\forall n \geq N} ~
    \tau_{n}^{2}\mu_{n} 
    \leq \tau_{N}^{2} \mu_{N}  
    +\rbr{2\gamma}^{-1}\norm{u_{N}}^{2}
  \end{math}
  or, equivalently,
  by the very definition of $\fa{\mu_{n}}{n \in \NPP} $,
  \begin{equation}
    \label{eq:limit.i}
    \rbr{\forall n\geq N} 
    \quad
    h\rbr{x_{n}} 
    \leq h\rbr{z} + \frac{\tau_{N}^{2}}{\tau_{n}^{2}}\mu_{N}
    + \frac{1}{2\gamma \tau_{n}^{2}} \norm{u_{N}}^{2}.
  \end{equation}
  Consequently, since $\tau_{n} \uparrow
  \pinf$,
  taking the
  limit superior in \cref{eq:limit.i}
  gives
  \begin{math}
    \inf_{n\geq N} h\rbr{x_{n}} 
    \leq \varlimsup h\rbr{x_{n}} 
    \leq h\rbr{z} ,
  \end{math}
  which contradicts \cref{eq:z-existence}. 

  (b) $\tau_{\infty} < \pinf$:  
  Set $\rbr{\forall n\geq N} ~ 
  \xi_{n}\coloneqq \tau_{n}^{2}\mu_{n}
  + \rbr{2\gamma}^{-1}\norm{u_{n}}^{2}
  \text{ and } \eta_{n}\coloneqq
  \rbr{\tau_{n}^{2}-\tau_{n+1}^{2} + \tau_{n+1}}\mu_{n}$.
  Then, 
  by \cref{eq:positive}, 
  $\set{\xi_{n}}{n\geq N} \subset \RPP$
  and, by \cref{eq:z-existence}\&\cref{eq:rewrite-cont-tau},
  $\set{\eta_{n}}{n\geq N} \subset \RP$.
  In turn, 
  on the one hand,
  combining \cref{eq:induction.general}
  and \cref{l:summable-limit}\cref{i:summable}, 
  we infer that 
  \begin{math}
    \sum_{n\geq N}\rbr{\tau_{n}^{2} -\tau_{n+1}^{2}+\tau_{n+1}} \mu_{n}
    = \sum_{n\geq N}\eta_{n} < \pinf.
  \end{math}
  On the other hand, 
  because $\rbr{\forall n \in \NPP} ~ \tau_{n}^{2} \leq \rbr{\sup_{k \in \NPP}\tau_{k}}^{2}<\pinf$
  and $\set{\tau_{n}}{n \in \NPP} \subset \left[1,\pinf \right[$
      by our assumption and \cref{eq:taun.geq1},
  \begin{equation}
    \rbr{\forall p\in \NPP} 
    \quad 
    \sum_{n=N}^{N+p}\rbr[\big]{\tau_{n}^{2}-\tau_{n+1}^{2} + \tau_{n+1}} 
    = \tau_{N}^{2} - \tau_{N+p+1}^{2} + \sum_{n=N}^{N+p}\tau_{n+1}
    \geq \tau_{N}^{2} - \rbr[\Big]{\sup_{n \in \NPP}\tau_{n}}^{2} 
    + p+1,
  \end{equation}
  from which we deduce that $\sum_{n\geq N}\rbr{\tau_{n}^{2}-\tau_{n+1}^{2}+\tau_{n+1}} = \pinf$.
  Altogether, \cref{l:summable-liminf} and \cref{eq:positive}
  guarantee that 
  \begin{math}
    \varliminf \rbr{h\rbr{x_{n}}  - h\rbr{z} } 
    =\varliminf \mu_{n}=0,
  \end{math}
  i.e., $\varliminf h\rbr{x_{n}} = h\rbr{z} $.
  Consequently, due to
  the inequality
  $\inf_{n\geq N} h\rbr{x_{n}}  \leq  
  \varliminf h\rbr{x_{n}} $,
  it follows from \cref{eq:z-existence}
  that $h\rbr{z} < h\rbr{z} $,
  which is absurd.

  To summarize, we have reached a contradiction
  in each case, and therefore
  \cref{eq:prop.establish} holds.
  Thus, because 
  \begin{math}
    \min_{1\leq k \leq n}h\rbr{x_{k}} 
    \to \inf_{m \in \NPP}h\rbr{x_{m}}
  \end{math}
  as $n\to \pinf$,
  we infer from 
  \cref{eq:prop.establish} that 
  \begin{math}
    \min_{1\leq k \leq n}h\rbr{x_{k}} 
    \to \inf h
  \end{math}
  as $n\to \pinf$. 
  Finally, 
  \cref{eq:prop.establish}
  guarantees that 
  $\varliminf h\rbr{x_{n}} 
  = \sup_{n \in \NPP}\rbr[\big]{\inf_{k\geq n} h\rbr{x_{k}}  } 
  = \sup_{n \in \NPP}\rbr{ \inf h} 
  = \inf h$,
  which completes the proof.
\end{proof}

\begin{remark}
  \label{rm:open.prob1}
  In \cref{t:general.case},
  we do not know whether or not $\fa{h\rbr{x_{n}} }{n \in \NPP} $
  converges to $\inf h$.
However,  \cref{t:main-fista}, \cref{t:main.ista}, 
  and \cref{p:limit-inf}
  suggest a positive answer.
\end{remark}

We are now ready for our second main result (\cref{t:main-fista}),
which is a discrete version
of Attouch et al.'s \cite[Theorem~2.3]{Attouch-fast-MPB-2018}.
When $\fa{\tau_{n}}{n \in \NPP}$
is as in \cref{eg:seq}\cref{i:seq1}
with $\rho=2$, 
items \cref{i:main1}
and \cref{i:main3}
were  mentioned 
(without a detailed proof)
in \cite[Theorem~4.1]{Attouch-fast-arxiv-2015}. 
The analysis of \cref{t:main-fista}\cref{i:main2}
was motivated by 
Attouch and Cabot's \cite[Proposition~3]{Attouch-Cabot-HAL2017}.
Furthermore, the boundedness 
of the sequences $\fa{x_{n}}{n \in \NPP} $
and $\fa{n\norm{x_{n}- x_{n-1}} }{n \in \NPP} $
in the consistent case was first obtained
in Attouch et al.'s \cite[Proposition~4.3]{Attouch2017rate};
here, we slightly modified the
proof of this result to obtain
the boundedness of $\fa{x_{n}}{n \in \NPP} $
in a more general setting.

\begin{theorem}
	\label{t:main-fista}
    Suppose that 
    \begin{equation}\label{eq:assump-main.fista}
      \inf{h} >\minf
      \quad
      \text{and}
      \quad
      \sup_{n \in \NPP}\rbr{n/\tau_{n}} < \pinf.
    \end{equation}
    For every $z \in \dom h$,
    set 
    $\beta_{z} \coloneqq \tau_{1}^{2} \rb{h\rb{x_{1}} - h\rb{z} } + 
        \rb{2\gamma}^{-1}\norm{ \tau_{1}x_{1} - \rb{\tau_{1}-1}x_{0}-z}^{2}$. 
	Then the following hold:
	\begin{enumerate}
          \item\label{i:main0} For every $z \in \dom{h}$, 
            we have 
            \begin{equation}
              \label{eq:main-todo2}
              \tau_{n}^{2}\rbr{h\rbr{x_{n}} -h\rbr{z} }
              + \rbr{2\gamma}^{-1}\norm{\tau_{n}x_{n}-\rbr{\tau_{n}-1} x_{n-1}-z}^{2}
              \leq \beta_{z} + \tau_{n}^{2}\rbr[\big]{h\rbr{z} -\inf h}
              \sup_{k \in \NPP}\rbr{k/\tau_{k}} 
            \end{equation}
            and 
            \begin{equation}
              \label{eq:main-todo}
              \rb{\forall n \in \NPP} 
              \quad 
              h\rb{x_{n}} - h\rb{z} \leq 
              \frac{\beta_{z}}{\tau_{n}^{2}} 
              + \rb[\big]{h\rb{z} - \inf{h}}\sup_{k \in \NPP}\rb{k/\tau_{k}}.
            \end{equation}	
          \item\label{i:main1} \begin{math}
		h\rb{x_{n}}\to  \inf {h}.
		\end{math}
		\item\label{i:main2} 
		$\fa{x_{n}}{n \in \NPP}$
		is asymptotically regular, i.e., 
		$x_{n} - x_{n-1}  \to 0$.
		\item\label{i:main3}  Every weak
		sequential cluster point of
		$\fa{x_{n}}{n \in \NPP}$
		belongs to $\Argmin{h}$.
		\item\label{i:main4} Suppose that 
		$\fa{x_{n}}{n \in \NPP}$
        has a bounded subsequence.
		Then $\Argmin{h} \neq \varnothing$.
		\item\label{i:main5} Suppose that 
		$\Argmin{h} = \varnothing$.
		Then $\norm{x_{n}} \to +\infty$.
        \item \label{i:main6} 
          Suppose that $\Argmin{h} \neq \varnothing$.
        Then the following hold:  
        \begin{enumerate}
          \item\label{i:main.6a}  
            \emph{(Beck--Teboulle \cite{BeckTeboulle-FISTA})}
          $h\rbr{x_{n}} -\min h = \bigO{\rbr{1/n^{2}}} $
          as $n\to \pinf$; more precisely, 
          for every $z \in \Argmin h$,
          \begin{equation}
            \label{eq:classic}
            \rb{\forall n \in \NPP}
            \quad 
            h\rb{x_{n}} - \min{h}
            \leq 
            \frac{\beta_{z}\rbr[\big]{\sup_{k \in \NPP}\rbr{k/\tau_{k}} }^{2} }{n^{2}}.
          \end{equation}
          \item\label{i:main.6b}
            The sequences $\fa{x_{n}}{n \in \NPP} $
            and $\fa{\tau_{n}\rbr{x_{n}-x_{n-1}} }{n \in \NPP} $
            are
            bounded. 
        \end{enumerate}
	\end{enumerate}
\end{theorem}

\begin{proof} 

Set $\kappa \coloneqq \sup_{n \in \NPP}\rb{n/\tau_{n}} \in \RPP$.
Since $\rb{\forall n \in \NPP}~\tau_{n} \geq  n/\kappa$,
we see that 
\begin{equation}
  \label{eq:taun.blowsup}
  \tau_{n} \to +\infty.
\end{equation}

\cref{i:main0}: 
Take $z \in \dom{h}$,
and set
\begin{equation}\label{e:vn-un}
\rb{\forall n \in \NPP }
\quad 
\mu_{n} \coloneqq h\rb{x_{n}} - h\rb{z}
\quad 
\text{and}
\quad 
u_{n} \coloneqq \tau_{n} x_{n} - \rb{\tau_{n}-1}x_{n-1} - z.
\end{equation}
Now, for every $n \in \NPP $, 
since $\inf{h} > -\infty$, 
$\tau_{n} \leq \tau_{n+1},$
and $\tau_{n+1} \rb{\tau_{n+1} -1} \leq \tau_{n}^{2}$, 
applying \cref{l:1-fista-step}\cref{i:1step-3} 
to $\rb{y,x_{-}, \tau ,\tau_{+}} = \rb{y_{n},x_{n-1},\tau_{n},\tau_{n+1} }$
yields 
\begin{math}
\tau_{n+1}^{2}\mu_{n+1}
+ \rb{2\gamma}^{-1}\norm{ u_{n+1} }^{2}
\leq
\tau_{n}^{2}\mu_{n}
+ \rb{2\gamma}^{-1}\norm{ u_{n} }^{2}
+\tau_{n+1} \rb{h\rb{z} - \inf{h} }.
\end{math}
Hence, 
because $\fa{\tau_{n}}{n \in \NPP}$
is increasing
and $h\rb{z} - \inf{h} \geq 0$, 
an inductive argument gives 
\begin{subequations}
  \label{eq:induct}
	\begin{align}
	\rb{\forall n \in \NPP}
	\quad
	\tau_{n+1}^{2}\mu_{n+1}
	+ \rb{2\gamma}^{-1}\norm{ u_{n+1} }^{2} 
    & \leq
	\tau_{1}^{2}\mu_{1}
	+\rb{2\gamma}^{-1}\norm{u_{1}}^{2} 
    +\rb[\big]{h\rb{z} - \inf{h}}\sum_{k=2}^{n+1}\tau_{k}  \\
	& \leq \tau_{1}^{2}\mu_{1}
	+\rb{2\gamma}^{-1}\norm{u_{1}}^{2} 
    +  n \tau_{n+1} \rbr[\big]{h\rb{z}- \inf{h}}  \\
    & \leq  \beta_{z} 
    + \kappa \tau_{n+1}^{2} \rb[\big]{h\rb{z} - \inf{h} } .
	\end{align}
\end{subequations}
Therefore,
since \cref{eq:main-todo2}
trivially holds when $n=1$,
we obtain the conclusion.
Consequently, \cref{eq:main-todo}
readily follows from \cref{eq:main-todo2}.

\cref{i:main1}: 
For every $z \in \dom{h}$,
taking the limit superior over $n$
in \cref{eq:main-todo}
and using \cref{eq:taun.blowsup} yields 
\begin{math}
\varlimsup h\rb{x_{n}}\leq  h\rb{z} + \kappa \rb{h\rb{z} - \inf{h}}.
\end{math}
Consequently, letting $h\rb{z} \downarrow \inf{h}$,
we conclude that $\varlimsup h\rb{x_{n}} \leq \inf{h}$,
as desired.

\cref{i:main2}: 
First, due to \cref{e:defn-en},
$\rb{\forall n \in \NPP}~
\sigma_{n} \geq h\rb{x_{n}} \geq \inf{h} > \minf$,
and thus,
\begin{equation}
  \label{eq:inf.sigma}
  \inf_{n \in \NPP}\sigma_{n} > \minf.
\end{equation}
Hence,
we conclude
via \cref{l:decreasing}\cref{i:decrease.ii}
that 
$\fa{\sigma_{n}}{n \in \NPP}$
is convergent in $\RR$. 
In turn, 
on the one hand, 
\cref{i:main1} and \cref{e:defn-en}
imply that 
\begin{equation}
  \label{eq:converges}
  \fa[\big]{ \norm{x_{n} -x_{n-1}}^{2} }{n \in \NPP}
  \text{~converges in } \RR.
\end{equation}
On the other hand, 
\cref{eq:inf.sigma} and 
\cref{l:decreasing}\cref{i:sum.2a} yield 
$\sum_{n \in \NPP}\rb{1-\alpha_{n}^{2}}\norm{x_{n}-x_{n-1}}^{2} < \pinf$,
and since $\sum_{n \in \NPP}\rb{1-\alpha_{n}^{2}} = \pinf$
due to \cref{l:blowsup}
and \cref{eq:taun.blowsup}, 
we get from \cref{l:summable-liminf} that 
\begin{equation}
  \label{eq:liminf.0}
 \varliminf \norm{x_{n}-x_{n-1}}^{2} = 0. 
\end{equation}
Altogether,  
combining \cref{eq:converges} 
and \cref{eq:liminf.0}
yields $x_{n}-x_{n-1}\to 0$, 
as announced.

\cref{i:main3}: Let $x$
be a weak sequential cluster point
of $\fa{x_{n}}{n\in \NPP}$,
say $x_{k_{n}}\rightharpoonup x$.
Then, since $h$
is 
convex and 
lower semicontinuous, it is
weakly sequentially lower semicontinuous
by \cite[Theorem~9.1]{Bauschke-Combettes-2017}.
Hence, \cref{i:main1}
entails that $h\rb{x} \leq \varliminf h\rb{x_{k_{n}}} = \inf{h} $,
which ensures that $x \in \Argmin{h}$.

\cref{i:main4}: Combine \cref{i:main3} and 
\cite[Lemma~2.45]{Bauschke-Combettes-2017}.

\cref{i:main5}: This is 
the contrapositive of \cref{i:main4}.

\cref{i:main.6a}: Clear from \cref{i:main0}
and \cref{eq:assump-main.fista}. 

\cref{i:main.6b}: Fix $z \in \Argmin h$.
For every $n \geq 2$, 
because $h\rbr{z} = \min h$, 
we derive from \cref{eq:main-todo2} that 
\begin{equation}
  \label{eq:expand}
    \rbr{2\gamma}^{-1} \norm{\tau_{n}x_{n}-\rbr{\tau_{n}-1} x_{n-1}-z}^{2} 
    \leq \tau_{n}^{2}\rbr{h\rbr{x_{n}} - \min h} + 
    \rbr{2\gamma}^{-1} \norm{\tau_{n}x_{n}-\rbr{\tau_{n}-1} x_{n-1}-z}^{2} 
    \leq \beta_{z},
\end{equation}
and now a simple expansion gives
\begin{subequations} 
  \begin{align}
    2\gamma\beta_{z} 
    & \geq  \norm{\tau_{n}x_{n}-\rbr{\tau_{n}-1}x_{n-1} -z }^{2} \\
    & =  \norm{\rbr{x_{n}-z} +\rbr{\tau_{n}-1}\rbr{x_{n}-x_{n-1}}   }^{2} \\
    & = \norm{x_{n}-z}^{2}
    + 2\rbr{\tau_{n}-1} \scal{x_{n}-z}{x_{n}-x_{n-1}} + 
    \rbr{\tau_{n}-1}^{2}\norm{x_{n}-x_{n-1}}^{2} \\
    & \geq \norm{x_{n}-z}^{2} 
    +2\rbr{\tau_{n}-1} \scal{x_{n}-z}{x_{n}-x_{n-1}  } \\
    &\overset{\cref{e:points}}{=} \norm{x_{n}-z}^{2} +
    \rbr{\tau_{n}-1} \rbr*{ \norm{x_{n}-z}^{2}
     - \norm{x_{n-1}-z}^{2} + \norm{x_{n}-x_{n-1}}^{2}  } 
    \\
    & = \tau_{n}\norm{x_{n}-z}^{2} - \rbr{\tau_{n}-1}\norm{x_{n-1}-z}^{2}
    + \rbr{\tau_{n}-1} \norm{x_{n}-x_{n-1}}^{2}
    \\
    & \overset{\cref{eq:difference}}{\geq}
    \tau_{n}\norm{x_{n}-z}^{2} - \tau_{n-1}\norm{x_{n-1}-z}^{2}.
  \end{align}
 \end{subequations}
 In turn, 
 \begin{equation}
   \label{eq:sum.bounded}
   \rbr{\forall n\geq  2} ~
   \tau_{n}\norm{x_{n}-z}^{2} 
   - \tau_{1}\norm{x_{1}-z}^{2}
   = \sum_{k=2}^{n}\rbr[\big]{\tau_{k}\norm{x_{k}-z}^{2}
   -\tau_{k-1}\norm{x_{k-1}-z}^{2}} 
   \leq \sum_{k=2}^{n}2\gamma\beta_{z}
   \leq  2\gamma\beta_{z}n.
 \end{equation}
 Hence, since $\kappa = \sup_{n \in \NPP}\rbr{n/\tau_{n}} < \pinf$
 and $\rbr{\forall n \in \NPP} ~\tau_{1}\leq\tau_{n}$,
 we get 
  \begin{equation}
    \label{eq:bounded.xn}
    \rbr{\forall n \geq 2} \quad
    \norm{x_{n}-z}^{2} 
    \leq 2\gamma\beta_{z}\frac{n}{\tau_{n}}
    + \frac{\tau_{1}}{\tau_{n}}\norm{x_{1}-z}^{2}
    \leq 2\gamma\beta_{z}\kappa + \norm{x_{1}-z}^{2},
  \end{equation}
  from which the boundedness of  
  $\fa{x_{n}}{n \in \NPP} $ follows. 
  Consequently,
  because $\fa{\tau_{n}\rbr{x_{n}-x_{n-1}} }{n \in \NPP} 
  =  \fa{\tau_{n}x_{n}-\rbr{\tau_{n}-1} x_{n-1}-z}{ n\in \NPP}  -
  \fa{x_{n-1}-z}{n \in \NPP} $
  and both sequences on
  the right-hand side are bounded due to
  \cref{eq:expand} and \cref{eq:bounded.xn},
  we conclude that $\fa{\tau_{n}\rbr{x_{n}-x_{n-1}} }{n \in \NPP} $
  is bounded,
  as announced.
\end{proof}

\begin{remark}
By choosing
the sequence $\fa{\tau_{n}}{n \in\NPP} $
as in \cref{eg:seq}\cref{i:seq2}, 
we shall see in \cref{p:limit-inf}
that \cref{t:main-fista}\cref{i:main1}
is still valid even when the assumption that $\inf{h}> - \infty$ is omitted.
Therefore, it is appealing to conjecture that 
this assumption can be left out in  \cref{t:main-fista}\cref{i:main1}.
In stark contrast,
it is crucial to assume that $h$ is bounded from below
in \cref{t:main-fista}\cref{i:main2},
as illustrated in 
\cref{eg:counter.eg}. 
\end{remark}

\begin{example}\label{eg:counter.eg}
	Suppose that $\HH = \RR$, 
	that $f \colon \HH \to \RR:  x \mapsto -x$, 
	that $g = 0$,
	that $\gamma =1$,
	and that 
    $\tau_{n}\uparrow \tau_\infty = \pinf $.
	Then, since $\prox{g} = \Id$
    and 
    $\rb{\forall x \in \HH}~\nabla f\rb{x} = -1$,
    we see that 
	\cref{e:FISTA-algo} turns into
		\begin{align}
		& \text{for~} n=1,2,\ldots \notag \\
		& \left\lfloor
		\begin{array}{ll}
		x_{n} & \coloneqq 
		y_{n}  + 1, \medskip  \\ 
		y_{n+1} & \displaystyle  
		\coloneqq 
        x_{n} + \frac{\tau_{n}-1}{\tau_{n+1}}
		\rb{x_{n} - x_{n-1}}.
		\end{array}
      \right.\label{eq:countereg}
		\end{align}
	Hence, \begin{math}
		\rb{\forall n \in \NPP}
		~ x_{n+1} -1
		= y_{n+1} 
        = x_{n} + \rb{\tau_{n}-1}\rb{x_{n}-x_{n-1}}/\tau_{n+1},
	\end{math}
	and upon setting $\rb{\forall n\in \NPP}~z_{n}\coloneqq x_{n}-x_{n-1}$,
	we obtain 
		\begin{equation}\label{e:zn}
			\rb{\forall n \in \NPP}
			\quad 
            z_{n+1} = 1+ \frac{\tau_{n}-1}{\tau_{n+1}}z_{n}.
		\end{equation}
	Let us establish that 
		\begin{math}
          z_{n} \to \pinf.
		\end{math}
        First, since $y_{1}=x_{0}$
        by \cref{alg:FISTA},
        we get from \cref{eq:countereg}
        that 
        $z_{1} =x_{1}-x_{0} = x_{1}-y_{1} = 1 $.
    In turn, by induction and \cref{e:zn},
    $\rbr{\forall n \in \NPP} ~z_{n}\geq 1$.
    We now suppose to the contrary that 
    $\xi \coloneqq \varliminf z_{n} \in \RP$.
    Then, taking the limit inferior over $n$
    in \cref{e:zn}
    and using \cref{l:quotient} yield
    $\xi = 1 + 1\cdot \xi = 1 + \xi$,
    which is absurd.
    Therefore, $\xi = \pinf$,
    and it follows that 
    $x_{n} - x_{n-1} = z_{n}\to \pinf$.
\end{example}

\begin{proposition}\label{p:limit-inf}
  Suppose that 
	the sequence $\fa{\tau_{n}}{n \in \NPP}$
    is as in \cref{eg:seq}\cref{i:seq2}.
	Then 
         $h\rb{x_{n}} \to \inf{h}\in\left[\minf,\pinf\right[$.
\end{proposition}

\begin{proof}
	First, as seen in \cref{eg:seq}\cref{i:seq2},
	\begin{equation}\label{e:seq-original}
		\rb{\forall n \in \NPP}
		\quad \tau_{n+1}^{2} - \tau_{n+1} = \tau_{n}^{2}.
	\end{equation}
 	Now it is sufficient 
	to show that 
	\begin{math}
		\varlimsup h\rb{x_{n}} \leq \inf{h}.
	\end{math}
	To do so, 
	fix $z \in \dom{h}$,
	and set  
\begin{math}
\rb{\forall n \in \NPP }
~ 
\mu_{n} \coloneqq h\rb{x_{n}} - h\rb{z}
\text{~and~}
u_{n} \coloneqq \tau_{n} x_{n} - \rb{\tau_{n}-1}x_{n-1} - z.
\end{math}
	Then, according to 
	\cref{l:1-fista-step}\cref{i:1step-2}
	and \cref{e:seq-original},  
		\begin{equation}
		\rb{\forall n \in \NPP}
		\quad 
		\tau_{n+1}^{2}\mu_{n+1}
		+\rb{2\gamma}^{-1}\norm{u_{n+1}}^{2}
		 \leq \tau_{n+1}\rb{\tau_{n+1}-1}\mu_{n}
		+ \rb{2\gamma}^{-1}\norm{u_{n}}^{2} 
		 = \tau_{n}^{2} \mu_{n}  +\rb{2\gamma}^{-1}\norm{u_{n}}^{2}.
		\end{equation}
	Thus, 
	\begin{equation}
	\rb{\forall n \in \NPP}
	\quad 
		h\rb{x_{n}} - h\rb{z} = \mu_{n} \leq  
		\frac{\tau_{n}^{2}\mu_{n} + 
			\rb{2\gamma}^{-1}\norm{u_{n}}^{2}}{\tau_{n}^{2} }  
		\leq \frac{\tau_{1}^{2}\mu_{1} +
			 \rb{2\gamma}^{-1}\norm{u_{1}}^{2}}{\tau_{n}^{2} }.
	\end{equation}
	Hence, 
	because $
	\lim \tau_{n} = +\infty$, 
	taking 
	the limit superior 
	over $n$ yields 
	$\varlimsup h\rb{x_{n}} \leq   h\rb{z}  $.
	Consequently,
	since $z$ is an arbitrary element
	of $\dom{h}$,
	we conclude that 
	\begin{math}
		\varlimsup h\rb{x_{n}}  \leq \inf{h},
	\end{math}
	as required.
\end{proof}

\begin{remark}
  \label{rm:Silvia}
  \cref{p:limit-inf}
  is a special case of the 
  \emph{accelerated inexact forward-backward
  splitting} developed in \cite{Villa-Salzo-Luca-Verri-2013};
  see \cite[Theorem~4.3 and Remark~3]{Villa-Salzo-Luca-Verri-2013}.
\end{remark}

We now turn to our third main result, which concerns
the case where the parameter sequence 
$\fa{\tau_{n}}{n \in \NPP}$ in \cref{assump:2} is bounded.

\begin{theorem}
  \label{t:main.ista}
  Suppose that $\tau_{\infty} < \pinf$.
  Then the following hold:
  \begin{enumerate}
    \item\label{i:ista.1} $\lim \sigma_{n} = \lim h\rbr{x_{n}} = \inf h \in \left[\minf,\pinf \right[$. 
    \item\label{i:ista.2} Assume that $\inf h >\minf$.
      Then the following hold: 
      \begin{enumerate}
        \item \label{i:ista.2a}
          $\sum_{n \in \NPP}\norm{x_{n}-x_{n-1}}^{2} < \pinf$.
    \item\label{i:ista.2b}
      Every weak sequential cluster point
      of $\fa{x_{n}}{n \in \NPP}$ lies in $\Argmin h$.
      \end{enumerate}
    \item\label{i:ista.3} Assume that $\fa{x_{n}}{n \in \NPP} $
      has a bounded subsequence. Then $\Argmin h \neq\varnothing$.
    \item\label{i:ista.4} Assume that $\Argmin h = \varnothing$. Then
      $\norm{x_{n}} \to \pinf$.
    \item\label{i:ista.5} Assume that $\Argmin h \neq \varnothing$.
      Then the following hold: 
      \begin{enumerate}
        \item\label{i:ista.5a} 
          \emph{(Attouch--Cabot \cite{Attouch-Cabot-HAL2017})}
        $h\rbr{x_{n}} - \min h = \smallO{\rbr{1/n}} $
      as $n\to \pinf$.
    \item\label{i:ista.5ab}
      \begin{math}
        \sum_{n \in \NPP}n \norm{x_{n}-x_{n-1}}^{2} < \pinf.
      \end{math}
      As a consequence,
      $\norm{x_{n}-x_{n-1}} = \smallO{\rbr{1/\sqrt{n}} }$
      as $n\to \pinf$.
      \end{enumerate}
  \end{enumerate}
\end{theorem}

\begin{proof}
  \cref{i:ista.1}: 
  Since, by \cref{e:defn-en}, 
  $\rbr{\forall n \in \NPP} ~
  \inf h \leq h\rbr{x_{n}} \leq \sigma_{n}$
  and, by \cref{l:decreasing}\cref{i:decrease.ii}, 
  $\fa{\sigma_{n}}{n \in \NPP} $ converges to
  a point $\sigma \in \left[\minf,\pinf \right[$,
  it is enough to verify that 
  $\sigma = \lim  \sigma_{n} = \inf h$. 
  Assume to the contrary that 
  \begin{equation}
    \label{eq:ista.contradiction}
    \minf \leq \inf h< \sigma.
  \end{equation}
  It then follows that $\inf_{n \in \NPP}\sigma_{n} > \minf$,
  and \cref{l:decreasing}\cref{i:sum.2b}
  thus yields  $\norm{x_{n}-x_{n-1}}^{2} \to 0$,
  from which and \cref{e:defn-en}
  we deduce that $h\rbr{x_{n}} \to \sigma$.
  This and \cref{t:general.case}
  imply that $\sigma = \inf h$. 
  This and
  \cref{eq:ista.contradiction} yield a contradiction.

  \cref{i:ista.2a}: Our assumption 
  ensures that $\inf_{n \in \NPP}\sigma_{n} > \minf$,
  and therefore, thanks to the 
  boundedness of $\fa{\tau_{n}}{n \in \NPP}$,
  \cref{l:decreasing}\cref{i:sum.2b} yields 
  $\sum_{n \in \NPP}\norm{x_{n}-x_{n-1}}^{2} < \pinf$.

  \cref{i:ista.2b}\&\cref{i:ista.3}\&\cref{i:ista.4}: Similar
  to \cref{t:main-fista}\cref{i:main3}\&\cref{i:main4}\&\cref{i:main5},
   respectively.

   \cref{i:ista.5}:
Fix $z \in \Argmin h$,
   and set 
   $\rbr{\forall n \in \NPP} ~\mu_{n}\coloneqq h\rbr{x_{n}} -h\rbr{z}  = 
   h\rbr{x_{n}} -\min h\geq 0 
   \text{ and } 
   u_{n}\coloneqq \tau_{n} x_{n} - \rbr{\tau_{  n}  -1}x_{n-1}-z$.
   By \cref{eq:lim.tau}, 
   we have 
   \begin{equation}
     \label{eq:taun.kappa}
     \tau_{n} \uparrow \tauinf,
   \end{equation}
   which implies that  
   $\tau_{n}^{2} - \tau_{n+1}^{2}+\tau_{n+1} \to \tauinf$.
   Therefore, 
   because $\tauinf \in \RPP$,
   there exists $N \in \NPP$ 
   such that    
   \begin{equation}
     \label{eq:inf.ista}
     \inf_{n \geq N}\rbr[\big]{\tau_{n}^{2} - \tau_{n+1}^{2}+\tau_{n+1}} \geq \frac{\tauinf}{2}. 
   \end{equation}
   Next, for every $n \geq N$, using 
   \cref{l:1-fista-step}\cref{i:1step-2}
   with $\rbr{y,x_{-},\tau,\tau_{+}}  = 
   \rbr{y_{n},x_{n-1},\tau_{n},\tau_{n+1}} $,
   we get 
   $\tau_{n+1}^{2}\mu_{n+1} + \rbr{2\gamma}^{-1}\norm{u_{n+1}}^{2}
   \leq \tau_{n}^{2}\mu_{n} + \rbr{2\gamma}^{-1}\norm{u_{n}}^{2} 
   -\rbr{\tau_{n}^{2}-\tau_{n+1}^{2} +\tau_{n+1}} \mu_{n}$.
   Hence, because
   $\set{\tau_{n}^{2}\mu_{n} + \rbr{2\gamma}^{-1}\norm{u_{n}}^{2} }{n \geq N} \subset\RP$
   and, by \cref{eq:rewrite-cont-tau}, 
   $\set{\rbr{\tau_{n}^{2}-\tau_{n+1}^{2} +\tau_{n+1}} \mu_{n}}{n\geq N} \subset\RP$,
   \cref{l:summable-limit}\cref{i:summable}
   and \cref{eq:inf.ista} give 
   \begin{math}
     \rbr{\tauinf/2} \sum_{n\geq N}\mu_{n} \leq 
     \sum_{n\geq N}\rbr{\tau_{n}^{2}-\tau_{n+1}^{2} + \tau_{n+1}}\mu_{n}
     < \pinf.
   \end{math}
   This, \cref{i:ista.2a},
   and \cref{e:defn-en} ensure that 
   \begin{equation}
     \label{eq:summable}
   \sum_{n \in \NPP}\rbr{\sigma_{n}- \min h } =
   \sum_{n \in \NPP}\rbr[\big]{\mu_{n} + \rbr{2\gamma}^{-1}\norm{x_{n}-x_{n-1}}^{2}}  
   < \pinf.
 \end{equation}
 Furthermore,
   \cref{l:decreasing}\cref{i:decrease.ii}
   and \cref{i:ista.1}
   yield 
   \begin{equation}
     \label{eq:down.to.0}
     \sigma_{n}-\min h \downarrow 0.
   \end{equation}

   \cref{i:ista.5a}: 
   Appealing to 
   \cref{eq:summable}
   and \cref{eq:down.to.0},
   \cref{l:seq2}
   guarantees that $n\rbr{\sigma_{n}- \min h} \to 0$.
   Consequently,
   since $\rbr{\forall n \in \NPP} ~
   \sigma_{n} -\min h = \rbr{h\rbr{x_{n}} - \min h} + \rbr{2\gamma}^{-1}
   \norm{x_{n}-x_{n-1}}^{2}
   \geq h\rbr{x_{n}} -\min h \geq 0$,
   the conclusion follows.

   \cref{i:ista.5ab}: 
    Thanks to \cref{eq:summable}
    and \cref{eq:down.to.0},
    we derive from \cref{l:seq2}
    that 
    \begin{equation}
      \sum_{n \in \NPP}n\rbr{ \sigma_{n}-\sigma_{n+1} } 
      = \sum_{ n \in \NPP} n \sbrc*{ \rbr{\sigma_{n}- \min h} 
      - \rbr{\sigma_{n+1}- \min h} }
      < \pinf,
    \end{equation}
    and hence,
    by \cref{l:decreasing}\cref{i:decrease.i},
    $\sum_{n \in \NPP}n\rbr{1-\alpha_{n}^{2}}\norm{x_{n}-x_{n-1}}^{2}
    < \pinf $.
    Thus, 
    because $\inf_{n \in \NPP}\rbr{1-\alpha_{n}^{2}} > 0$
    due to the boundedness of $\fa{\tau_{n}}{n \in \NPP} $
    and \cref{l:decreasing}\cref{i:sum.2b},
    we conclude that 
    $\sum_{n \in \NPP}n \norm{x_{n}-x_{n-1}}^{2} < \pinf$.
    This gives
    $n\norm{x_{n}-x_{n-1}}^{2} \to 0$, i.e.,
    $\norm{x_{n}-x_{n-1}} = \smallO{\rbr{1/\sqrt{n}} }$
    as $n\to \pinf$,
    as desired.
\end{proof}

\begin{remark}
  \label{rm:rm.ista}
  \ 
  \begin{enumerate}
    \item In the case of the classical
      forward-backward algorithm 
      (without the extrapolation step)
      with linesearches,
      results similar to \cref{t:main.ista}\cref{i:ista.1}\&\cref{i:ista.4}
      were established in \cite[Theorem~4.2]{Cruz-Nghia-2016} by Bello Cruz and Nghia.
      To the best of our knowledge,
      \cref{t:main.ista}\cref{i:ista.1} 
      is new in the setting of \cref{alg:FISTA}. 
    \item \cref{t:main.ista}\cref{i:ista.5a}
      was obtained by Attouch and Cabot \cite[Corollary~20(iii)]{Attouch-Cabot-HAL2017}.
      Here we provide
      a proof based on the technique developed
      in \cite{Attouch-Cabot-HAL2017}
      to be self-contained.
    \item 
      The summabilities 
      established in 
      \cref{t:main.ista}\cref{i:ista.2a}\&\cref{i:ista.5ab}
      are new. Nevertheless, 
      in the case of the forward-backward algorithm,
      i.e., when $\tau_{n} \equiv 1$, 
\cref{t:main.ista}\cref{i:ista.5ab}
appears
      implicitly in the Beck and Teboulle's proof of
      \cite[Theorem~3.1]{BeckTeboulle-FISTA}.
  \end{enumerate}
\end{remark}

In the case of
the classical forward-backward
algorithm, 
by applying \cite[Corollary~1.5]{Bruck-Reich}
to the forward-backward 
operator $\prox{\gamma g}\circ \mathop{ \rbr{\Id -
\mathop{\gamma {\grad{f} } } }}  $,
we obtain further
information on the sequence $\fa{x_{n}}{n \in \NPP} $
as follows.

\begin{proposition}
  \label{p:ISTA}
  Suppose that $\rbr{\forall n \in \NPP} ~ \tau_{n}=1$,
  and set\footnote{For
  a nonempty set $C$,
$\pr_{C}$ denotes the projector associated with $C$.}\footnote{The set $\closu{\ran\rbr{\Id -T} }$
is closed and convex by 
\cite[Corollary~4.2]{Moursi2018}
and 
\cite[Lemma~4]{Pazy1971}.} 
 $v\coloneqq \pr_{\closu{\ran \rbr{{\Id} - T } }}0 $.
  Then $x_{n}- x_{n-1}\to v$.
\end{proposition}

\begin{proof}
  By assumption, \cref{alg:FISTA}
  becomes $\rbr{\forall n\in \NPP} ~ x_{n} = T^{n}x_{0}$.
  Next,
  we learn from \cite[Proposition~3.2 and Corollary~4.2]{Moursi2018} that 
  $T$ is \emph{averaged}, i.e., there exists
$\alpha \in \left]0,1 \right[$
  and a nonexpansive operator $R \colon \HH \to \HH$
  such that $T = \rbr{1-\alpha} \Id + \alpha R$.
  Hence, we conclude via \cite[Proposition~1.3 and
  Corollary~1.5]{Bruck-Reich}
  that $x_{n} - x_{n-1} = T^{n}x_{0} - T^{n-1}x_{0}
  \to v$.
  For an alternative proof of \cite[Corollary~1.2]{Bruck-Reich}
  in the Hilbert space setting, see \cite[Proposition~2.1]{Moursi2018}. 
\end{proof}

\begin{remark}
Some comments are in order.
  \begin{enumerate}
    \item In stark contrast
  to \cref{p:ISTA}
  and \cref{t:main.ista}, 
  if $\tau_{\infty}=\pinf$,
  then it may happen that
  $\norm{x_{n}-x_{n-1}} \to \pinf$
  (see \cref{eg:counter.eg}).
\item For a recent study on
  the forward-backward operator $T$,
  we refer the reader to \cite{Moursi2018}.
  \end{enumerate}
\end{remark}

\begin{proposition}
  \label{p:convergence}
  Suppose that
  $\Argmin h\neq \varnothing$,
  that 
  $\fa{\tau_{n}^{2} \rbr{ h\rbr{x_{n}} - \min h  } }{n \in \NPP} $
  converges in $\RR$,
  and that 
  $\tau_{n}\norm{x_{n}-x_{n-1}} \to 0$.
  Then the following hold:
  \begin{enumerate}
    \item \label{i:convergence.1bis}
      $h\rbr{x_{n}} \to \min h$.
    \item\label{i:convergence.1} The sequence $\fa{x_{n}}{n \in \NPP} $
      converges weakly to a point in $\Argmin h$.
    \item\label{i:convergence.2} Suppose that $\inte\rbr{\Argmin h} \neq \varnothing$.
      Then $\fa{x_{n}}{n \in \NPP} $
      converges strongly to a point in $\Argmin h$.
  \end{enumerate}
\end{proposition}

\begin{proof}
  Set 
  \begin{equation}
    \label{eq:zn.converges}
    \rbr{\forall n \in \NPP} 
    \quad
    z_{n} \coloneqq \tau_{n}x_{n} -\rbr{\tau_{n}-1}x_{n-1}
    \text{ and }
    \varepsilon_{n}\coloneqq 2\gamma
    \rbr[\big]{\tau_{n}^{2}\rbr{h\rbr{x_{n}} - \min h} 
    - \tau_{n+1}^{2}\rbr{h\rbr{x_{n+1}} -\min h}}. 
  \end{equation}
  Since, 
  by \cref{eq:taun.geq1} and \cref{eq:zn.converges},
  $\rbr{\forall n \in \NPP} ~ \norm{z_{n}- x_{n}} 
  = \rbr{\tau_{n}-1} \norm{x_{n}-x_{n-1}} 
  \leq \tau_{n}\norm{x_{n}-x_{n-1}} $
  and since, by our assumption,
  $\tau_{n}\norm{x_{n}-x_{n-1}} \to 0$,
  we see that 
  \begin{equation}
    \label{eq:zn.xn}
    z_{n}-x_{n}\to 0.
  \end{equation}
  Next,
  due to our assumption and  
  \begin{subequations}
    \begin{align}
    \rbr{\forall n \in \NPP} \quad
    \sum_{k=1}^{n}\varepsilon_{k}
    & =
    2\gamma\sum_{k=1}^{n}\sbrc[\big]{ \tau_{k}^{2}\rbr{h\rbr{x_{k}} - \min h} 
  -\tau_{k+1}^{2}\rbr{h\rbr{x_{k+1}} - \min h } } 
    \\
    & =2\gamma\rbr[\big]{\tau_{1}^{2}\rbr{h\rbr{x_{1}}- \min h } 
    - \tau_{n+1}^{2}\rbr{h\rbr{x_{n+1}} -\min h} } ,
  \end{align}
  \end{subequations}
  we see that 
  \begin{equation}
    \label{eq:en.converges}
    \sum_{n \in \NPP}\varepsilon_{n} 
    \text{~is convergent in } \RR.
  \end{equation}
  Let us now establish that 
  \begin{equation}
    \label{eq:establish}
    \rbr{\forall z \in \Argmin h}\rbr{\forall n \in \NPP} 
    \quad
  \norm{z_{n+1}-z}^{2} \leq \norm{z_{n}-z}^{2} +\varepsilon_{n}. 
  \end{equation}
  Fix $z \in \Argmin h$ and $n \in \NPP$.
  Applying \cref{l:1-fista-step}\cref{i:1step-2}
  to $\rbr{y,x_{-},\tau,\tau_{+}} = \rbr{y_{n},x_{n-1},\tau_{n},\tau_{n+1}} $
  and invoking \cref{eq:rewrite-cont-tau}
  yields 
  \begin{subequations}
    \begin{align}
    \tau_{n+1}^{2}\rbr{h\rbr{x_{n+1}} - \min h} 
    +\rbr{2\gamma}^{-1}\norm{z_{n+1}-z}^{2}
    & \leq \tau_{n+1}\rbr{\tau_{n+1}-1} 
    \underbrace{ \rbr{h\rbr{x_{n}} - \min h}}_{\geq 0}
    + \rbr{2\gamma}^{-1}\norm{z_{n}-z}^{2} \\
    & \leq \tau_{n}^{2}\rbr{h\rbr{x_{n}} - \min h} 
    +\rbr{2\gamma}^{-1}\norm{z_{n}-z}^{2},
  \end{align}
  \end{subequations}
  from which and \cref{eq:zn.converges}
  we obtain 
  \cref{eq:establish}.

  \cref{i:convergence.1bis}:
Since, by assumption, 
  $\fa{\tau_{n}^{2} \rbr{ h\rbr{x_{n}}  - \min h }  }{ n \in \NPP} $
  converges and since, by \cref{eq:lim.tau},
  $\fa{1/\tau_{n}^{2}}{n \in \NPP} $ converges in $\RR$, 
  it follows that 
  $\fa{h\rbr{x_{n}} - \min h}{ n \in \NPP} $ is convergent
  in $\RR$.
  Therefore, due to \cref{t:general.case},
  $h\rbr{x_{n}} - \min h \to 0$.

  \cref{i:convergence.1}: 
  In the light of  \cref{i:convergence.1bis}, 
  arguing similarly 
  to the proof of  \cref{t:main-fista}\cref{i:main3},
we conclude that
  \begin{equation}
    \label{eq:weak.cluster}
    \text{every weak sequential cluster point of }
    \fa{x_{n}}{n \in \NPP} 
    \text{ belongs to } \Argmin h.
  \end{equation}
  In turn, 
  appealing to 
  \cref{eq:en.converges} and \cref{eq:establish},
\cref{l:fejer}\cref{i:fejer.1} implies that 
\begin{equation}
  \label{eq:norm.zn}
  \rbr{\forall z \in \Argmin h} \quad
  \fa{\norm{z_{n}-z} }{n \in \NPP} \text{ is convergent in } \RR.
\end{equation}
  Thus, combining \cref{eq:zn.xn}\&\cref{eq:weak.cluster}\&\cref{eq:norm.zn},
  we get via \cref{l:opial.variant}
  that $\fa{x_{n}}{n \in \NPP}$
  converges weakly to a point in $\Argmin h$.

  \cref{i:convergence.2}: 
  Since $\inte\rbr{\Argmin h}  \neq \varnothing$,
  owing to \cref{l:fejer}\cref{i:fejer.2},
  we derive from \cref{eq:en.converges} and \cref{eq:establish}
  that 
  there exists $z \in \HH$
  such that 
  $z_{n } \to z$. Hence, 
  by \cref{eq:zn.xn},
  $x_{n}\to z$, and 
  \cref{i:convergence.1} implies that  
  $z \in \Argmin h$.
  To sum up, $\fa{x_{n}}{n \in \NPP} $ converges
  strongly to a minimizer of $h$. 
\end{proof}

\begin{corollary}
  \label{c:ista.convergent}
 Suppose that $\Argmin h \neq \varnothing$
 and that $\sup_{n \in \NPP}\tau_{n} < \pinf$.
 Then $\fa{x_{n}}{n \in \NPP} $ converges weakly
 to a point in $\Argmin h$. 
 Moreover, if $\inte\rbr{\Argmin h}  \neq \varnothing$,
 then the convergence is strong.
\end{corollary}

\begin{proof}
  By \cref{t:main.ista}\cref{i:ista.5},
  we see that $h\rbr{x_{n}} - \min h \to 0$
  and $\norm{x_{n}-x_{n-1}} \to 0$,
  and since $\sup_{n \in \NPP}\tau_{n} < \pinf$,
  it follows that $\tau_{n}^{2}\rbr{h\rbr{x_{n}} - \min h} \to 0$
  and $\tau_{n}\norm{x_{n}-x_{n-1}} \to 0$.
  The conclusion thus follows from
  \cref{p:convergence}.
\end{proof}

\begin{remark}
    Consider the setting of 
    \cref{c:ista.convergent}.
    Although the weak convergence of
    the sequence $\fa{x_{n}}{n \in \NPP} $
    has been shown in \cite[Corollary~20(iv)]{Attouch-Cabot-HAL2017},
    our Fej\'{e}r-based proof here is new
    and may suggest other approaches to 
    tackle the convergence of $\fa{x_{n}}{n \in \NPP} $
    in the setting \cref{t:main-fista}\cref{i:main6}.
\end{remark}

We conclude this section with an instance where
the assumption of \cref{p:convergence} holds.

\begin{example}
  Suppose, in addition to \cref{assump:2},
that there exists $\delta \in \left]0,1 \right[$
  such that 
  \begin{equation}
    \label{eq:Attouch.cond}
  \rbr{\forall n \in \NPP} \quad 
  \tau_{n+1}^{2}-\tau_{n}^{2} \leq \delta \tau_{n+1}
\end{equation}
(see \cref{eg:Attouch-cond,eg:Aujol-seq}). 
Then Attouch and Cabot's \cite[Theorem~9]{Attouch-Cabot-HAL2017}
yields $\tau_{n}^{2}\rbr{h\rbr{x_{n}} - \min h} \to 0$
and $\tau_{n}\norm{x_{n}-x_{n-1} } \to 0 $.
\end{example}

\section{MFISTA}
\label{mfista}

In this section,
we discuss the minimizing
property of the sequence
generated by MFISTA.
The monotonicity of function values
allows us to overcome the issue 
stated in \cref{rm:open.prob1}.
Compared to Beck and Teboulle's \cite[Theorem~5.1]{BeckTeboulle-MFISTA}
(see also \cite[Theorem~10.40]{BeckSIAM}),
we allow other possibilities for the
choice of $\fa{\tau_{n}}{n \in \NPP} $ in \cref{t:mfista}\cref{i:mfista.6}.
Furthermore,
we provide in item \cref{i:mfista7}, which was motivated by
\cite[Theorem~9]{Attouch-Cabot-HAL2017},
a better rate of convergence.

\begin{theorem}
  \label{t:mfista}
  In addition to \cref{assump:2}, 
  suppose that $\tauinf= \pinf$.
  Let $x_{0} \in \HH$, set $y_{1}\coloneqq x_{0}$,
  and update
  \begin{align}
    & \text{for~} n=1,2,\ldots\notag \\ 
    &\label{eq:mfista.alg}
    \left\lfloor
    \begin{array}{ll}
      z_{n} & \coloneqq Ty_{n}, \smallskip \\
      x_{ n} & 
          \coloneqq \begin{cases}
            x_{n-1}, &\text{if~} h\rbr{x_{n-1}} \leq h\rbr{z_{n}} ; \\
            z_{n}, & \text{otherwise},
          \end{cases} \smallskip \\ 
          y_{n+1} 
          & \displaystyle 
          \coloneqq x_{n} + \frac{\tau_{n}}{\tau_{n+1}}\rbr{z_{n}-x_{n}} 
          + \frac{\tau_{n}-1}{\tau_{n+1}}\rbr{x_{n}-x_{n-1}} ,
    \end{array}
  \right.
  \end{align}
  where $T$ is as in \cref{eq:T.and.h}.
  Furthermore, set 
\begin{equation}
  \label{eq:mfista.sigma}
  \rbr{\forall n \in \NPP} \quad
  \sigma_{n}\coloneqq h\rbr{x_{n}}  + \frac{1}{2\gamma}\norm{z_{n}-x_{n-1}}^{2}.
\end{equation}
Then the following hold:
\begin{enumerate}
  \item\label{i:mfista.1} 
    \begin{math}
      \fa{h\rbr{x_{n}} }{n \in \NPP} 
    \end{math}
    is decreasing and 
\begin{math}
  h\rbr{x_{n}}  \downarrow \inf h.
\end{math}
  \item\label{i:mfista.2} 
    \begin{math}
      \fa{\sigma_{n}}{n \in \NPP} 
    \end{math}
    is decreasing and 
\begin{math}
  \sigma_{n} \downarrow \inf h.
\end{math}
\item\label{i:mfista.3} Suppose that $\inf h > \minf$. Then $z_{n} - x_{n-1}\to 0$
  and $x_{n}-x_{n-1}\to 0$.
\item\label{i:mfista.4} Suppose that $\fa{x_{n}}{n \in \NPP} $ 
  has a bounded subsequence. Then $\Argmin h \neq\varnothing$.
\item\label{i:mfista.5} Suppose that $\Argmin h = \varnothing.$ Then $\norm{x_{n}} \to \pinf$.
\item\label{i:mfista.6} Suppose that $\Argmin h \neq \varnothing.$ 
  Then $h\rbr{x_{n}}- \min h =\bigO{\rbr{1/\tau_{n}^{2}}}$
  as $n\to \pinf$.
\item\label{i:mfista7}
Suppose that $\Argmin h\neq\varnothing$
and that there exists $\delta\in\left]0,1\right[$ such that
\begin{equation}
\label{e:Attouch.cond2}
(\forall n\in\NPP)\quad
\tau_{n+1}^{2}-\tau_n^2\leq\delta\tau_{n+1}.
\end{equation}
Then
\begin{equation}
\label{e:mfista-rate1}
h(x_n)-\min h=\smallO{\bigg(\frac{1}{\sum_{k=1}^n\tau_k}\bigg)}
\text{ as }n\to\pinf
\end{equation}
and
\begin{equation}
\label{e:otn-mfista.1}
h(x_n)-\min h=\smallO{\bigg(\frac{1}{\tau_n^2}\bigg)}
\text{ as }n\to\pinf.
\end{equation}
\end{enumerate}
\end{theorem}

\begin{proof}
  \cref{i:mfista.1}: 
    By \cref{eq:mfista.alg}, 
    the sequence $\fa{h\rbr{x_{n}} }{n \in \NPP} $
    is decreasing, from which we have
    $h\rbr{x_{n}} \downarrow \inf_{k \in \NPP}h\rbr{x_{k}} $.
    Therefore,
    it suffices to prove that 
    $\inf_{n \in \NPP} h\rbr{x_{n}}  = \inf h$.
    To this end, assume
    to the contrary that 
    $\inf_{n \in \NPP} h\rbr{x_{n}}  > \inf h $.
    This yields the existence of a point
    $w \in \dom h$ such that 
    \begin{equation}
      \label{eq:mfista.w}
    \inf_{n \in \NPP}h\rbr{x_{n}} > h\rbr{w} .
    \end{equation}
     Set 
    \begin{equation}
      \label{eq:mfista.mu.u}
      \rbr{\forall n \in \NPP} \quad
      \mu_{n}\coloneqq h\rbr{x_{n}} -h\rbr{w} 
      \quad
      \text{and}
      \quad 
      u_{n}\coloneqq \tau_{n}z_{n} -\rbr{\tau_{n}-1} x_{n-1}
      -w.
    \end{equation}
    In turn, 
    for every $n \in \NPP$,
    because,
    by \cref{eq:rewrite-cont-tau}, 
    $ \tau_{n+1}\rbr{\tau_{n+1}-1} \leq \tau_{n}^{2}$
    and, by \cref{eq:mfista.w},  
    $\mu_{n}>  0$,
    it 
    follows from \cref{l:MFISTA.1step}\cref{i:mfista.1stepb}
    (applied to $\rbr{y,x_{-},\tau,\tau_{+}} = \rbr{y_{n},x_{n-1},\tau_{n},\tau_{n+1}}  $)
    that 
    \begin{equation}
\label{e:mfista-estimate}
        \tau_{n+1}^{2}\mu_{n+1} + \rbr{2\gamma}^{-1}\norm{u_{n+1}}^{2} 
         \leq  \tau_{n+1}\rbr{\tau_{n+1}-1} \mu_{n} + \rbr{2\gamma}^{-1}
        \norm{u_{n}}^{2} \\
         \leq \tau_{n}^{2}\mu_{n} + \rbr{2\gamma}^{-1}\norm{u_{n}}^{2}.
    \end{equation}
Hence,
    \begin{equation}
      \rbr{\forall n\in \NPP} \quad
      h\rbr{x_{n}} - h\rbr{w} = \mu_{n} \leq 
      \frac{1}{\tau_{n}^{2}}\rbr*{\tau_{n}^{2}\mu_{n} +\rbr{2\gamma}^{2}\norm{u_{n}}^{2}}
      \leq\frac{1}{\tau_{n}^{2}}\rbr*{\tau_{1}^{2}\mu_{1} +\rbr{2\gamma}^{2}\norm{u_{1}}^{2}}.
      \label{eq:mfista.lim}
    \end{equation}
    Consequently, since $h\rbr{x_{n}} \downarrow \inf_{k\in \NPP}h\rbr{x_{k}} $
    and $\tau_{n}\to \pinf$,
    we derive from \cref{eq:mfista.lim} that 
\begin{math}
  \inf_{n \in \NPP}h\rbr{x_{n}} \leq h\rbr{w},
\end{math}
which contradicts \cref{eq:mfista.w}.

\cref{i:mfista.2}: Let us first
show that $\fa{\sigma_{n}}{n \in \NPP} $
is decreasing.
Towards this end, 
for every $ n\in \NPP$,
we deduce from
\cref{l:MFISTA.1step}\cref{i:mfista.1stepa}
that 
\begin{math}
  \sigma_{n+1}
  = h\rbr{x_{n+1}} +\rbr{2\gamma}^{-1}\norm{z_{n+1}-x_{n}}^{2}
  \leq h\rbr{x_{n}} + \rbr{2\gamma}^{-1}\tau_{n}^{2}\norm{z_{n}-x_{n-1}}^{2}/\tau_{n+1}^{2}
  = \sigma_{n}- \rbr{2\gamma}^{-1}\rbr{1-\tau_{n}^{2}/\tau_{n+1}^{2}}
  \norm{z_{n}-x_{n-1}}^{2}.
\end{math}
Therefore, 
\begin{equation}
  \label{eq:mfista.sum}
  \rbr{\forall n \in \NPP} \quad
  \frac{1}{2\gamma}\rbr*{1-\frac{\tau_{n}^{2}}{\tau_{n+1}^{2}}}\norm{z_{n}-x_{n-1}}^{2} 
  \leq \sigma_{n}-\sigma_{n+1},
\end{equation}
and because $\rbr{\forall n \in \NPP} ~ 0< \tau_{n}/\tau_{n+1}\leq 1$,
we conclude that 
\begin{equation}
  \label{eq:sigma.decrease}
  \fa{\sigma_{n}}{n \in \NPP} \text{~is decreasing}.
\end{equation}
It remains to show that $\sigma_{n}\to \inf h$.  
Set $\sigma\coloneqq \inf_{n \in \NPP}\sigma_{n}$. 
Due to \cref{eq:sigma.decrease},
\begin{equation}
  \label{eq:sigma.lim}
  \sigma_{n}\downarrow \sigma
\end{equation}
and it therefore suffices to prove that $\sigma = \inf h$. 
Let us argue by contradiction: assume that $\sigma > \inf h \geq \minf$.
By \cref{eq:mfista.sum}, 
\begin{equation}
  \label{eq:bounded.mfista}
  \rbr{\forall n \in \NPP} 
  \quad
  \frac{1}{2\gamma}\sum_{k =1}^{n} 
  \rbr*{1-\frac{\tau_{k}^{2}}{\tau_{k+1}^{2}}} \norm{z_{k}-x_{k-1}}^{2}
  \leq \sum_{k=1}^{n}\rbr{\sigma_{k} -\sigma_{k+1}}
  = \sigma_{1} - \sigma_{n+1}
  \leq \sigma_{1} -\sigma < \pinf,
\end{equation}
which implies that $\sum_{n \in \NPP}\rbr{1-\tau_{n}^{2}/\tau_{n+1}^{2}}
\norm{z_{n}-x_{n-1}}^{2} < \pinf$.
Thus, since $\sum_{n \in \NPP}\rbr{1-\tau_{n}^{2}/\tau_{n+1}^{2}} = \pinf$
by \cref{l:blowsup},
\cref{l:summable-liminf} guarantees that 
$\varliminf \norm{z_{n}-x_{n-1}}^{2} =0$, i.e.,
$\varliminf \norm{z_{n}-x_{n-1}} = 0$.
In turn, let $\fa{k_{n}}{n \in \NPP} $ 
be a strictly increasing sequence in $\NPP$
such that $\norm{z_{k_{n}} - x_{k_{n}-1}} \to 
\varliminf \norm{z_{n}-x_{n-1}}= 0$.
It follows from \cref{i:mfista.1} and \cref{eq:sigma.lim}
that
\begin{math}
  \sigma \leftarrow \sigma_{k_{n}} 
  = h\rbr{x_{k_{n}}}  + \rbr{2\gamma}^{-1}\norm{z_{k_{n}}-x_{k_{n}-1}}^{2}   
  \to \inf h + 0 = \inf h.
\end{math}
Consequently, $\sigma =\inf h$, which violates 
the assumption that $\sigma > \inf h$.
To summarize, we have shown that 
$\sigma_{n}\downarrow \inf h$.

\cref{i:mfista.3}: 
Since $\inf h > \minf$,
combining \cref{i:mfista.1}, \cref{i:mfista.2},
and \cref{eq:mfista.sigma}
gives $z_{n} -x_{n-1} \to 0$.
To show that $x_{n}-x_{n-1}\to 0$, we
 infer from \cref{eq:mfista.alg}
that,
for every $n \in \NPP$,
$x_{n}-x_{n-1}=x_{n-1}-x_{n-1} = 0$ if 
$h\rbr{x_{n-1}} \leq h\rbr{z_{n}} $,
and $x_{n}-x_{n-1} = z_{n}-x_{n-1}$ otherwise;
therefore, $\rbr{\forall n \in \NPP} ~ \norm{x_{n}-x_{n-1}} \leq  
\norm{z_{n}-x_{n-1}}  $.
Consequently, because $z_{n}-x_{n-1}\to 0$,
it follows that $x_{n}-x_{n-1}\to 0$,
as required.

\cref{i:mfista.4}\&\cref{i:mfista.5}: Straightforward.

\cref{i:mfista.6}: 
Fix $w \in \Argmin h$ and define 
\begin{math}
  \rbr{\forall n \in \NPP} ~ \mu_{n}\coloneqq
  h\rbr{x_{n}} - h\rbr{w} 
  =h\rbr{x_{n}} -\min h \geq 0
  \text{ and }
  u_{n}\coloneqq \tau_{n}z_{n} -\rbr{\tau_{n}-1} x_{n-1}-w.
\end{math}
 Due to \cref{eq:rewrite-cont-tau}
and the fact that $\set{\mu_{n}}{n \in \NPP} \subset \RP$,
\cref{l:MFISTA.1step}\cref{i:mfista.1stepb}
entails that 
\begin{math}
  \rbr{\forall n \in \NPP} ~
  \tau_{n+1}^{2}\mu_{n+1} + \rbr{2\gamma}^{-1}\norm{u_{n+1}}^{2}
  \leq \tau_{n+1}\rbr{\tau_{n+1}-1} \mu_{n}+\rbr{2\gamma}^{-1}\norm{u_{n}}^{2}
  \leq \tau_{n}^{2}\mu_{n} + \rbr{2\gamma}^{-1}\norm{u_{n}}^{2}.
\end{math}
Hence, 
\begin{equation}
  \rbr{\forall n \in \NPP} \quad
  h\rbr{x_{n}} -\min h = \mu_{n} 
  \leq
\frac{1}{\tau_{n}^{2}}\rbr[\big]{\tau_{n}^{2}\mu_{n}+\rbr{2\gamma}^{-1}\norm{u_{n}}^{2}} 
  \leq
\frac{1}{\tau_{n}^{2}}\rbr[\big]{\tau_{1}^{2}\mu_{1}+\rbr{2\gamma}^{-1}\norm{u_{1}}^{2}}, 
\end{equation}
which verifies the claim.

\cref{i:mfista7}:
Let us adapt the notation of \cref{i:mfista.6}.
Since $\{\mu_{n}\}_{n\in\NPP}\subset\RP$,
we derive from \cref{l:MFISTA.1step}\cref{i:mfista.1stepb} and
\cref{e:Attouch.cond2} that
\begin{subequations}
\label{e:rate-mfista}
\begin{align}
(\forall n\in\NPP)\quad
\tau_{n+1}^2\mu_{n+1}+(2\gamma)^{-1}\norm{u_{n+1}}^{2}
&\leq\tau_{n+1}(\tau_{n+1}-1)\mu_{n}+(2\gamma)^{-1}\norm{u_{n}}^{2}
\\
&=\tau_{n}^{2}\mu_{n}+(2\gamma)^{-1}\norm{u_{n}}^{2}
-\big(\tau_{n}^{2}-\tau_{n+1}^{2}+\tau_{n+1}\big)\mu_{n}
\\
&\leq\tau_{n}^{2}\mu_n+(2\gamma)^{-1}\norm{u_{n}}^{2}
-(1-\delta)\tau_{n+1}\mu_n.
\end{align}
\end{subequations}
On the other hand, since $\delta\in\left]0,1\right[$
and $\{\mu_n\}_{n\in\NPP}\subset\RP$,
it follows that $\{(1-\delta)\tau_{n+1}\mu_n\}_{n\in\NPP}\subset\RP$.
Combining this, \cref{e:rate-mfista}, and
\cref{l:summable-limit}\cref{i:summable},
we infer that $(1-\delta)\sum_{n\in\NPP}\tau_{n+1}\mu_n<\pinf$.
In turn, since $(\tau_n)_{n\in\NPP}$ is increasing
and $1-\delta>0$,
it follows that $\sum_{n\in\NPP}\tau_{n}\mu_n<\pinf$.
Consequently, since $(\mu_n)_{n\in\NPP}$ is decreasing
due to \cref{i:mfista.1} and since clearly $\sum_{n\in\NPP}\tau_{n}=\pinf$,
\cite[Lemma~22]{Attouch-Cabot-HAL2017} ensures that
\begin{equation}
\label{e:otn-mfista}
h(x_n)-\min h
=\mu_n
=\smallO{\bigg(\frac{1}{\sum_{k=1}^n\tau_k}\bigg)}
\text{ as }n\to\pinf,
\end{equation}
which establishes \cref{e:mfista-rate1}.
In turn, we deduce from \cref{e:Attouch.cond2}, \cref{e:otn-mfista}, and
\cref{i:mfista.1} that
\begin{subequations}
\begin{align}
0\leq\tau_{n+1}^2(h(x_{n+1})-\min h)
&=(h(x_{n+1})-\min h)\bigg(\tau_{1}^2+
\sum_{k=1}^{n}\big(\tau_{k+1}^{2}-\tau_{k}^2\big)\bigg)
\\
&\leq(h(x_{n+1})-\min h)\bigg(\tau_{1}^{2}+\delta\sum_{k=1}^{n}\tau_{k+1}\bigg)
\\
&\leq(h(x_{n+1})-\min h)\bigg(\tau_{1}^{2}
-\delta\tau_1
+\delta\sum_{k=1}^{n+1}\tau_{k}\bigg)
\\
&\to 0\text{ as }n\to\pinf,
\end{align}
\end{subequations}
which verifies \cref{e:otn-mfista.1}.
\end{proof}

\begin{remark}
\label{r:last}
In \cref{t:mfista},
the assumption that $\tauinf=\pinf$ is actually not needed in items
\cref{i:mfista.1} and \cref{i:mfista.4}--\cref{i:mfista7}.
For clarity, let us sketch the proof
of \cref{i:mfista.1} under the assumption that
$\tau_\infty<\pinf$.
Assume that $\tau_\infty<\pinf$.
We infer from the first inequality in
\cref{e:mfista-estimate} that
\begin{equation}
(\forall n\in\NPP)\quad
\tau_{n+1}^2\mu_{n+1}+(2\gamma)^{-1}\norm{u_{n+1}}^2
\leq
\tau_n^2\mu_n+(2\gamma)^{-1}\norm{u_n}^2
-\big(\tau_n^2-\tau_{n+1}^2+\tau_{n+1}\big)\mu_n
\end{equation}
and it follows from \cref{l:summable-limit}\cref{i:summable}
that $\sum_{n\in\NPP}(\tau_n^2-\tau_{n+1}^2+\tau_{n+1})
\mu_n<\pinf$.
One may argue similarly to the case (b)
in the proof of \cref{t:general.case} to obtain
$\varliminf\mu_n=0$ or, equivalently, $\varliminf h(x_n)
=h(w)$, which contradicts \cref{eq:mfista.w}.
Therefore $\inf_{n\in\NPP}h(x_n)=\inf h$ and we get
$h(x_n)\downarrow\inf_{n\in\NPP}h(x_n)=\inf h$.
Items \cref{i:mfista.4} and \cref{i:mfista.5} follow
from this.
In addition,
note that we did not use the assumption that
$\tau_\infty=\pinf$ in the proof of
\cref{i:mfista.6} and \cref{i:mfista7}.
It is, however, worth pointing out that the conclusion of
\cref{t:mfista}\cref{i:mfista.6} is not so interesting when $\tauinf<\pinf$.
\end{remark}

\section{Open problems}
\label{openprobs}
  We conclude this paper
  with a few open problems.
  \begin{enumerate}[label=\textbf{P\arabic*},ref=\arabic*,leftmargin=*]
    \item In \cref{t:general.case},
      is it true that $h\rbr{x_{n}} \to \inf h$?
    \item 
  What can be said about
  the conclusions of \cref{t:main-fista}\cref{i:main2}\&\cref{i:main.6b}
  if $\sup_{n \in \NPP}\rbr{ n /\tau_{n}}  = \pinf$? 
\item Suppose that $\Argmin h \neq \varnothing$.
  Do the sequences   
  generated by \cref{alg:FISTA} and \cref{eq:mfista.alg} 
  always
  converge
  weakly to a point in $\Argmin h$?
  \end{enumerate}

\section*{Acknowledgments}
We thank two referees for
their very careful reading
and constructive comments.
	HHB and XW were partially 
	supported by NSERC Discovery Grants while 
	MNB was partially supported 
	by a Mitacs Globalink Graduate Fellowship Award.


\appendixpage

\begin{appendices}
  \crefalias{section}{appsec}

  \section{}\label{app:blowsup}

  For the sake of completeness, 
  we provide the following proof of \cref{l:blowsup}
  based on
  \cite[Problem~3.2.43]{Kaczor.Nowak-1}. 

    \begin{proof}[Proof of \cref{l:blowsup}]
    Because 
   $\rbr{\forall n\in \NPP} ~ 1-\rbr{\tau_{n}-1}^{2}/\tau_{n+1}^{2} 
   \geq 1- \tau_{n}^{2}/\tau_{n+1}^{2}$
   due to
   the assumption that
   $\set{\tau_{n}}{n \in \NPP} \subset \left[1,\pinf \right[ $,
       it is sufficient to establish that  
       \begin{equation}\label{eq:blowsup-prove}
         \sum_{n \in \NPP}\rbr*{1-\frac{\tau_{n}^{2}}{\tau_{n+1}^{2}}} 
         = \pinf.
       \end{equation}
    Indeed, 
    since $\tau_{n}\to \pinf$,
    there exists $N \in \NPP$
    such that 
    \begin{equation}
      \label{eq:N-blowsup}
      \rbr{\forall n \geq N} \quad 
      \tau_{n}^{2} \geq  2\tau_{1}^{2}.
    \end{equation}
    Now, 
    set $\rbr{\forall n \in \NPP} ~ \xi_{n} 
    \coloneqq \tau_{n+1}^{2}-\tau_{n}^{2}$,
    and 
    $\rbr{\forall n \in \NPP} ~ \sigma_{n}
    \coloneqq \sum_{k=1}^{n} \xi_{k}$.
    Then,
    on the one hand,
    since $\fa{\tau_{n}}{n \in \NPP}$
    is increasing and positive,
    we have 
    \begin{math}
      \rbr{\forall n \in \NPP} ~ 
      \xi_{n} = \tau_{n+1}^{2} - \tau_{n}^{2} \geq 0,
    \end{math}
    and $\fa{\sigma_{n}}{ n \in \NPP} $ is
    therefore an increasing sequence
    in $\RP$;
    moreover, due to \cref{eq:N-blowsup},
    \begin{math}
      \rbr{\forall n \geq N} ~ \sigma_{n} =
      \sum_{k=1}^{n}\rbr{\tau_{k+1}^{2}-\tau_{k}^{2}} =
      \tau_{n+1}^{2}-\tau_{1}^{2} \geq \tau_{1}^{2}
      \geq  1.
    \end{math}
    On the other hand, 
    because $\tau_{n}\to \pinf$,
    we have 
    $\sigma_{n}  =\tau_{n+1}^{2}-\tau_{1}^{2} \to \pinf$.
    Altogether, 
    since
    \begin{equation}
      \label{eq:cauchy-blowsup}
      \rbr{\forall n \geq  N} \rbr{\forall p \in \NPP} \quad
      \sum_{k=1}^{p}\frac{\xi_{n+k}}{\sigma_{n+k}}
      \geq  \sum_{k=1}^{p}\frac{\xi_{n+k}}{\sigma_{n+p}}
      = \frac{\sigma_{n+p}-\sigma_{n}}{\sigma_{n+p}}
      = 1-\frac{\sigma_{n}}{\sigma_{n+p}}  
    \end{equation}
    by the fact that $\fa{\sigma_{n}}{n \geq  N} $
    is increasing, 
   we see that 
   \begin{math}
     \rbr{\forall n \geq  N}~ 
     \varliminf_{p} \sum_{k=1}^{p}\rbr{ \xi_{n+k}/\sigma_{n+k}}
     \geq  1.
   \end{math}
    It follows that the partial sums
    of  
    \begin{math}
      \sum_{n \geq  N}\rbr{\xi_{n}/\sigma_{n}}
    \end{math}
    do not satisfy the Cauchy property.
    Hence, since $\rbr{\forall n \geq  N} ~ 
    \xi_{n}/\sigma_{n}\geq 0
    \text{ and }
    \sigma_{n}=\tau_{n+1}^{2}-\tau_{1}^{2}$, we obtain
   \begin{equation}
      \sum_{n \geq  N}\frac{\tau_{n+1}^{2}-\tau_{n}^{2}}{\tau_{n+1}^{2}-\tau_{1}^{2}}
      = \sum_{n \geq N}\frac{\xi_{n}}{\sigma_{n}} = \pinf.
    \end{equation}
    Consequently, in the light of 
    \cref{eq:N-blowsup},
    \begin{equation}
      \sum_{n \geq  N}\rbr*{1-\frac{\tau_{n}^{2}}{\tau_{n+1}^{2}}} 
      = \sum_{n \geq  N}\frac{\tau_{n+1}^{2}-\tau_{n}^{2}}{\tau_{n+1}^{2}}
      \geq  \sum_{n \geq  N}\frac{\tau_{n+1}^{2}-\tau_{n}^{2}}{2\rbr[\big]{\tau_{n+1}^{2}-\tau_{1}^{2}}}
      = \pinf,
    \end{equation}
    and \cref{eq:blowsup-prove} follows.
    \end{proof}

    \section{}\label{app:liminf-sum}
    \begin{proof}[Proof of \cref{l:summable-liminf}]
  Let us argue by contradiction.
  Towards this goal, assume that 
  $\varliminf \beta_{n} \in \left] 0, \pinf \right]$
  and fix $ \beta \in \left] 0 , \varliminf \beta_{n} \right[$.
  Then, there exists $N \in \NPP$ such that 
  $\rbr{\forall n \geq N} ~ \beta_{n} \geq \beta$,
  and hence, because $\set{ \alpha_{n} }{n \in \NPP} \subseteq \RP$,
  we have 
  $\rbr{\forall n \geq N} ~ \alpha_{n}\beta_{n} \geq \beta\alpha_{n}$.
  Consequently,
  since $\sum_{n \in \NPP}\alpha_{n} = \pinf$, 
  it follows that 
  $\sum_{n \geq N}\alpha_{n}\beta_{n} 
  \geq \sum_{n \geq  N} \beta \alpha_{n} = \pinf$,
  which violates our assumption.
  To sum up, $\varliminf \beta_{n} = 0$.
    \end{proof}

    \section{}\label{app:sum-limit}

    The following self-contained proof of \cref{l:summable-limit}
    follows 
    \cite[Lemma~3.1]{Combettes-fejer-2001} 
    in the case $\chi=1$;
    however, 
    we do not require the
    error sequence $\fa{\varepsilon_{n}}{n \in \NPP} $
    to be positive.

    \begin{proof}[Proof of \cref{l:summable-limit}]
		\ref{i:convergence}: Set $\alpha \coloneqq \varliminf_{n}\alpha_{n} 
            \in \left[\inf_{n \in \NPP}\alpha_{n}, \pinf\right]$ and let 
        $\fa{\alpha_{k_{n}}}{n \in \NPP}$ be a subsequence 
        of $\fa{\alpha_{n}}{n \in \NPP}$ that converges to 
        $\alpha$. 
        We first show that 
        $\alpha < \pinf$.
        Since $\set{\beta_{n}}{ n \in \NPP} \subset \RP$,
        it follows from \eqref{eq:seqs-cond} that 
        $\rbr{\forall n \in \NPP} ~ \alpha_{n+1} - \alpha_{n} \leq \varepsilon_{n }$. 
        Thus, $\rbr{\forall n \geq  2} ~ \alpha_{n}  
        =  \alpha_{1} +  \sum_{k=1}^{n-1}\rbr{\alpha_{k+1}-\alpha_{k}} 
        \leq \alpha_{1} + \sum_{k=1}^{n-1}\varepsilon_{k}$; 
        in particular, $\rbr{\forall n \geq  2} ~
        \alpha_{k_{n}}  \leq 
        \alpha_{1} +  \sum_{k=1}^{k_{n}-1}\varepsilon_{k}$.
        Hence, since $\alpha_{k_{n}}\to \alpha$
        and $\sum_{n \in \NPP}\varepsilon_{n}$ converges,
        it follows that
        $\alpha  \leq  \alpha_{1} +  \sum_{k \in \NN}\varepsilon_{k } < \pinf$, as claimed. 
        In turn, to establish the convergence of 
        $\fa{\alpha_{n}}{n \in \NPP} $,
        it suffices to verify that 
        $\varlimsup_{n}\alpha_{n} \leq \varliminf_{n} \alpha_{n}$.
        Towards this goal,
         let $\delta $ be in $\RPP$. 
        Then, on the one hand,  
        Cauchy's criterion ensures the existence of 
        $k_{n_{0}} \in \NPP$ 
        such that $\alpha_{k_{n_{0} } } - \alpha \leq \delta /2 $ and that 
        $\rbr{\forall n\geq k_{n_{0}}}\rbr{\forall m \in \NPP}~  
        \sum_{k=n}^{n+m}\varepsilon_{k } \leq \delta/2$.  
        On the other hand, 
        because $\set{\beta_{n}}{n \in \NPP} \subset\RP$,
        \eqref{eq:seqs-cond} implies that 
        $\rbr{\forall n \geq  k_{n_{0}} +1} ~ \alpha_{n} - \alpha_{k_{n_{0} } } 
        = \sum_{k=k_{n_{0}}}^{n-1} \rbr{ \alpha_{k+1} - \alpha_{k}} 
        \leq \sum_{k=k_{n_{0}}}^{n-1}\varepsilon_{k } $. 
        Altogether, $\rbr{\forall n \geq k_{n_{0}} +1} ~ \alpha_{n} 
        \leq \alpha_{k_{n_{0}}} + \sum_{k=k_{n_{0}}}^{n-1}\varepsilon_{k } 
        \leq \rbr{\alpha + \delta /2 } + \delta/2 
        = \alpha+ \delta$, from which we 
        deduce that $\varlimsup_{n} \alpha_{n} 
        \leq \alpha + \delta$. 
        Consequently, since $\delta$ is arbitrarily 
        chosen in $\RPP$, 
       it follows that  
        $\varlimsup_{n}\alpha_{n} 
        \leq \alpha = \varliminf_{n}\alpha_{n}$, and therefore, 
        $\fa{\alpha_{n}}{n \in \NN}$ converges to $\alpha$.
		
		\ref{i:summable}: We derive from 
        \eqref{eq:seqs-cond} that 
        $\rbr{\forall N \in \NPP} ~  
        \sum_{n=1}^{N}\beta_{n} \leq \sum_{n=1}^{N}\rbr{\alpha_{n} - \alpha_{n+1}} 
        + \sum_{n=1}^{N}\varepsilon_{n } 
        = \alpha_{1}-\alpha_{N+1} + \sum_{n=1}^{N}\varepsilon_{n }$. 
        Hence, since $\sum_{n \in \NPP}\varepsilon_{n}$ is convergent 
        and, by  \ref{i:convergence}, 
        $\lim_{n}\alpha_{n} =\alpha$,  
        letting $N \to \pinf$ yields 
        $\sum_{n \in \NN} \beta_{n} \leq \alpha_{1} - \alpha + \sum_{n \in \NN}\varepsilon_{n } <
        \pinf$, 
        and so $\sum_{n \in \NPP}\beta_{n} < \pinf$, as required.
    \end{proof}

    \section{}\label{app:seq1}
    \begin{proof}[Proof of \cref{l:seq.1}]
  Indeed, since
    $\rbr{\forall n \in \NPP}$
  \begin{subequations}
    \begin{align}
    \sum_{k=1}^{n}k\rbr{\alpha_{k}-\alpha_{k+1}} 
    & = \sum_{k=1}^{n}\rbr[\big]{k \alpha_{k} - \rbr{k+1}\alpha_{k+1}+\alpha_{k+1} } 
    \\
    & = \sum_{k=1}^{n}\rbr[\big]{k\alpha_{k}-\rbr{k+1} \alpha_{k+1} } 
    + \sum_{k=1}^{n}\alpha_{k+1}
     = \alpha_{1}- \rbr{n+1} \alpha_{n+1}
    +\sum_{k=1}^{n}\alpha_{k+1},
  \end{align}
  \end{subequations}
  we readily obtain the conclusion.
\end{proof}

\section{}\label{app:seq2}

\begin{proof}[Proof of \cref{l:seq2}]
  ``\ensuremath{\implies}'': 
  Since $\fa{\alpha_{n}}{n \in \NPP} $ is a decreasing
  sequence in $\RP$ and $\sum_{n \in \NPP}\alpha_{n} < \pinf$,
  it follows that 
   $n\alpha_{n}\to 0$
   (see, e.g., \cite[Problem~3.2.35]{Kaczor.Nowak-1}).
  Invoking the assumption that 
  $\sum_{n \in \NPP}\alpha_{n} < \pinf$ once more,
  we infer from \cref{l:seq.1} that 
  $\sum_{n \in \NPP}n\rbr{\alpha_{n}-\alpha_{n+1}} < \pinf$, 
  as desired.

  ``\ensuremath{\Leftarrow}'':
  A consequence of \cref{l:seq.1}.
\end{proof}

    \section{}\label{app:keylem}
    \begin{proof}[Proof of \cref{f:key-ineq}]
      This is similar to the one found 
  in \cite[Lemma~2.3]{BeckTeboulle-FISTA} and included for completeness; 
  see also \cite[Lemma~3.1]{Chambolle-Dossal-15}. 
      Fix $\rbr{x,y}  \in \HH\times\HH$.
      On the one hand, 
      by \cref{assump:f} and \cref{assump:gamma}
      in \cref{assump:1},
      $\grad{f}$ is Lipschitz continuous with constant $\gamma^{-1}$,
      from which, the Descent Lemma 
      (see, e.g., \cite[Lemma~2.64]{Bauschke-Combettes-2017}),
      and the convexity of $f$
      we infer that 
      \begin{subequations}
          \label{eq:key1}
        \begin{align}
          f\rbr{Ty}  &\leq
          f\rbr{y}  + \scal{\grad{f}\rbr{y} }{Ty -y}  +\rbr{2\gamma}^{-1}\norm{Ty-y}^{2}
          \\
          & = f\rbr{y} +\scal{\grad{f}\rbr{y} }{x-y}  
          + \scal{\grad{f}\rbr{y} }{Ty-x} + \rbr{2\gamma}^{-1}\norm{Ty-y}^{2} 
          \\
          & \leq f\rbr{x}  + \scal{\grad{f}\rbr{y} }{Ty-x} + \rbr{2\gamma}^{-1}\norm{Ty-y}^{2}. 
        \end{align}
      \end{subequations}
      On the other hand, 
      because $Ty =\prox{\gamma g}\rbr{y-\gamma{\grad{f}\rbr{y} }} $,
      \cite[Proposition~12.26]{Bauschke-Combettes-2017}
      asserts that 
      \begin{subequations}
        \label{eq:key2}
        \begin{align}
          g\rbr{Ty}  
        & \leq g\rbr{x} - \gamma^{-1}\scal{\rbr{y-\gamma{\grad{f}\rbr{y} }}- Ty }{x-Ty} 
        \\
        & \leq g\rbr{x}  + \scal{\gamma^{-1}\rbr{y-Ty} - {\grad{f}\rbr{y} }}{Ty-x}.
        \end{align}
      \end{subequations}
      Altogether, upon adding 
      \cref{eq:key1} and \cref{eq:key2},
      it follows that 
      \begin{subequations}
        \begin{align}
          h\rbr{Ty}  
          & \leq h\rbr{x}  + \gamma^{-1}\scal{y-Ty}{Ty-x} 
          +\rbr{2\gamma}^{-1}\norm{Ty-y}^{2}
          \\
          &  = h\rbr{x}  + \gamma^{-1}\scal{y-Ty}{y-x} 
          +\gamma^{-1}\scal{y-Ty}{Ty-y} +\rbr{2\gamma}^{-1}\norm{Ty-y}^{2} 
          \\
          & = h\rbr{x}  +\gamma^{-1}\scal{y-Ty}{y-x} 
          -\rbr{2\gamma}^{-1}\norm{Ty-y}^{2},
        \end{align}
      \end{subequations}
      which yields \cref{e:key-ineq}.
\end{proof}

\section{}
\label{app:folklore}
\begin{proof}[Proof of \cref{e:folklore}.]
Recall that $\lim(\tau_n/n)=1/2$.
In turn, because $(\forall n\in\NPP)$
$\tau_n^2=\tau_{n+1}^2-\tau_{n+1}$,
it follows that
\begin{equation}
\frac{n(\tau_n-\tau_{n+1})}{\tau_{n+1}}
=\frac{n(\tau_n^2-\tau_{n+1}^2)}{\tau_{n+1}(\tau_n+\tau_{n+1})}
=\frac{-n\tau_{n+1}}{\tau_{n+1}(\tau_n+\tau_{n+1})}
=\frac{-1}{\displaystyle
\frac{\tau_n}{n}+\frac{\tau_{n+1}}{n+1}\frac{n+1}{n}}
\to\frac{-1}{\frac{1}{2}+\frac{1}{2}}=-1
\end{equation}
and therefore that
\begin{equation}
n\Bigg(\frac{\tau_n-1}{\tau_{n+1}}-1+\frac{3}{n}\Bigg)
=\frac{n(\tau_n-\tau_{n+1})}{\tau_{n+1}}
-\frac{n+1}{\tau_{n+1}}\frac{n}{n+1}+3
\to -1-2+3=0.
\end{equation}
Hence, \cref{e:folklore} holds.
\end{proof}
\end{appendices}
\end{document}